\documentclass[final,12pt,3p]{elsarticle}

\usepackage[english]{babel}

\usepackage{amsmath}
\usepackage{amssymb}
\usepackage{graphicx}
\usepackage[colorlinks=false]{hyperref}
\usepackage{bm}
\usepackage{mathtools}
\usepackage{hhline}
\usepackage{amsthm}
\usepackage{subcaption}
\usepackage{xcolor}
\usepackage[normalem]{ulem}

\usepackage{listings}

\usepackage{algpseudocode}
\usepackage{algorithm}
\algnewcommand{\IIf}[1]{\State\algorithmicif\ #1\ \algorithmicthen}
\algnewcommand{\EndIIf}{\unskip\ \algorithmicend\ \algorithmicif}

\theoremstyle{remark} 
\newtheorem{remark}{Remark}

\theoremstyle{plain} 
\newtheorem{lemma}{Lemma}

\theoremstyle{plain}

\theoremstyle{definition}
\newtheorem{definition}{Definition}

\usepackage{setspace}
\AtBeginEnvironment{tabular}{\singlespacing}
\AtBeginEnvironment{algorithm}{\singlespacing}

\journal{Journal of Computational Physics}

\begin{document}

\begin{frontmatter}

\title{Fully-Discretely Nonlinearly-Stable Flux Reconstruction Methods for Compressible Flows}

\author[mcgill]{Carolyn M. V. Pethrick}
\ead{carolyn.pethrick@mail.mcgill.ca}

\author[mcgill]{Siva Nadarajah}
\ead{siva.nadarajah@mcgill.ca}

\affiliation[mcgill]{organization={Department of Mechanical Engineering, McGill University},
            addressline={845 Sherbrooke St W}, 
            city={Montreal},
            postcode={H3A 0G4}, 
            state={Quebec},
            country={Canada}}

\begin{abstract}
    A fully-discrete, nonlinearly-stable flux reconstruction (FD-NSFR) scheme is developed, which ensures robustness through entropy stability in both space and time for high-order flux reconstruction schemes.
    We extend the entropy-stable flux reconstruction semidiscretization of {Cicchino et al.~\cite{cicchino2022nonlinearly,cicchino2022provably,cicchino2023discretely}} with the relaxation Runge Kutta method to construct the FD-NSFR scheme.
    We focus our study on entropy-stable flux reconstruction methods, which allow a larger time step size than discontinuous Galerkin.
    In this work, we develop an FD-NSFR scheme that prevents temporal numerical entropy change in the broken Sobolev norm if the governing equations admit a convex entropy function that can be expressed in inner-product form.
    For governing equations with a general convex numerical entropy function, temporal entropy change in the physical $L_2$ norm is prevented. 
    As a result, for general convex numerical entropy, the FD-NSFR scheme achieves fully-discrete entropy stability only when the DG correction function is employed. 
    We use entropy-conserving and entropy-stable test cases for the Burgers', Euler, and Navier-Stokes equations to demonstrate that the FD-NSFR scheme prevents temporal numerical entropy change. 
    The FD-NSFR scheme therefore allows for a larger time step size while maintaining the robustness offered by entropy-stable schemes.
    We find that the FD-NSFR scheme is able to recover both integrated quantities and solution contours at a higher target time-step size than the semi-discretely entropy-stable scheme, suggesting a robustness advantage for low-Mach turbulence simulations.
\end{abstract}


\end{frontmatter}

\section{Introduction}
\label{introduction}

High-order (HO) methods have been developed to solve PDEs with high resolution, yielding more accurate results {\color{black}than low-order methods} at an equivalent number of degrees of freedom \cite{wang2013high}.
Industry demand for high-fidelity fluid simulation for large eddy simulation (LES) or direct numerical simulation (DNS) has motivated the development of HO methods. 
However, high computational cost and robustness concerns, especially in such high-cost simulations, have prevented the widespread adoption of HO methods.

The discontinuous Galerkin (DG) method is a finite element method that allows arbitrarily high convergence orders by using polynomial basis functions and localized, highly parallelizable structure by allowing discontinuities in the solution between elements~\cite{cockburn2001Runge}.
First introduced for the solution of the neutron transport equations~\cite{reed1973triangular}, DG has become a popular choice for high-fidelity fluids applications.
A drawback of DG is the prohibitively small time step size for linear stability, leading to the development of flux reconstruction (FR),  first proposed by Huynh~\cite{huynh_flux_2007}.
FR allows a larger time step size by adding a correction function to the flux~\cite{huynh_flux_2007}.
Wang and Gao~\cite{wang2009unifying} extended Huynh's approach to more element types in the lifting collocation penalty (LCP) approach.
The linearly stable versions of FR and LCP methods were generalized into energy-stable FR (ESFR) for linear advection by Vincent \textit{et al.}~\cite{vincent_new_2011}.
Within the ESFR framework, it is possible to recover a number of schemes whose trial space of functions are only piecewise continuous, such as DG, Huynh's FR, and spectral-difference by adjusting a single parameter~\cite{vincent_new_2011}.
The class of ESFR methods is also provably stable for linear diffusion equations. 
Castonguay~\cite{castonguay2013energy} proved stability for diffusion using the local discontinuous Galerkin flux.
This was extended to several compact numerical fluxes in 1D~\cite{quaegebeur2019stability} and 2D~\cite{quaegebeur2019stability2D} by  Quaegebeur, Nadarajah and co-authors.
Allaneau and Jameson~\cite{allaneau2011connections} showed that ESFR methods can be cast as filtered DG methods for linear advection on 1D elements. 
ESFR as filtered DG was further generalized for general nonlinear conservation laws on curvilinear elements in higher dimensions by Zwanenburg and Nadarajah~\cite{zwanenburg2016equivalence}.
ESFR can easily be implemented in an existing DG code using the filtered-DG approach.
While ESFR affords stability proofs for linear PDEs, there are no guarantees provided for stability in nonlinear cases. 
Nonlinear stability is an important issue for industrial applications, where the PDEs of interest can be highly nonlinear.

Entropy stability can address the problem of nonlinear stability in high-order methods.
Tadmor~\cite{tadmor1987numerical} proposed the construction of discretizations that discretely uphold the entropy inequality.
Such notions were first extended to HO methods by LeFloch~\cite{lefloch2002fully} for arbitrarily HO finite difference methods.
The first development of current split forms in the HO community was by Fisher \textit{et al.}~\cite{fisher2013discretely}, who constructed summation-by-parts (SBP) split operators on bounded domains. 
Gassner~\cite{gassner2013skew} applied SPB operators to construct entropy-stable DG schemes. 
Entropy-stable methods have been applied to FR by~\cite{ranocha2016summation,abe2018stable}, but with the requirement for collocated flux and solution nodes. 
Chan~\cite{chan2018discretely} developed a skew-symmetric stiffness operator that enabled entropy-stable DG on general choices of solution and flux nodes, provided that projections between nodes used the entropy variables.
Numerical experiments confirmed that choosing flux nodes with higher integration strength improved the robustness of solutions in the presence of density gradients~\cite{chan2022entropy}.
In this paper, we will use the scheme developed by Cicchino and co-authors~\cite{cicchino2022nonlinearly,cicchino2022provably,cicchino2023discretely}. Their nonlinearly-stable FR (NSFR) scheme employs splitting on the volume fluxes and uses entropy-projected variables to formulate an entropy-stable scheme for any ESFR variation.
The scheme demonstrates entropy stability for the Burgers' ~\cite{cicchino2022nonlinearly} and Euler equations~\cite{cicchino2023discretely} using uncollocated nodes, curvilinear meshes, and arbitrary FR correction functions while maintaining efficient scaling at $\mathcal{O}(p^{d+1})$. 
All of the aforementioned HO entropy-stable methods are formulated semi-discretely, such that a very small time step size is required for unsteady problems. While they provide a nonlinear stability guarantee, their applicability is limited by the requirement for an extremely small time step size.

As small time step sizes cause HO entropy-stable methods to be costly, it is attractive to develop temporal integration methods which consider entropy stability, thereby formulating fully-discrete (FD) entropy-stable schemes. 
One approach is to use space-time methods, as analyzed by Friedrich \textit{et al.}~\cite{friedrich2019entropy}. 
However, space-time DG methods are not trivial to implement, and geometry can be challenging in higher dimensions.
For more conventional method-of-lines strategies, the relaxation Runge-Kutta (RRK) method, introduced by Ketcheson~\cite{ketcheson2019relaxation} and expanded by Ranocha \textit{et al.}~\cite{ranocha2020relaxation}, allows many RK methods to be used to provide fully-discrete entropy stability. 
The method has been extended to multistep~\cite{ranocha2020general} and IMEX~\cite{kang2022entropy,li2022implicit} temporal methods. 
{Najafian and Vermeire~\cite{najafian2024energy} proposed the relaxation-free Runge-Kutta method as an alternative to Ketcheson's inner-product RRK. This method modifies the Butcher tableau instead of adjusting the time step size.}
To improve parallel performance, ~\cite{ranocha2020fully} introduced a local RRK method, but the approach lacks a global temporal entropy guarantee.
A performance review found that global RRK is approximately 1.5 times more expensive at the same time step size~\cite{aljahdali2022on,rogowski2022performance}, but can be justified considering the nonlinear stability properties afforded.
The RRK method has been applied to formulate fully-discrete schemes by other authors, whose results demonstrate robustness. 
Various entropy-stable spatial schemes have been used in conjunction with RRK to construct FD entropy-stable schemes, including but not limited to collocated DGSEM for turbulence simulations~\cite{parsani2021simulation}, DGFEM difference methods \cite{yan2023entropy}, and ADER-DG schemes~\cite{gaburro2023high}. 
In this work, we use the global variant of RRK~\cite{ketcheson2019relaxation, ranocha2020relaxation}.
This paper extends the semi-discrete NSFR scheme to a fully-discrete entropy-stable scheme.
We do not know of any authors who have applied the root-solving version of RRK to non-DG FR. This work is the first FD-NSFR in broken Sobolev norms for general entropy functions.
The results herein show fully-discrete entropy stability in the broken Sobolev norm for inner-product numerical entropy functions and nullification of temporal numerical entropy in the physical $L^2$ norm for general convex numerical entropy functions.

Questions have been raised in the community regarding the extent of stabilization provided by entropy-stable high-order schemes. While it can be proved that the solution will not diverge, even bounded oscillations may still corrupt the solution. Chan and co-authors~\cite{chan2022entropy} highlight that entropy-stable discretizations {\color{black}do not prevent density from becoming negative} in flows with strong gradients. While the entropy projection method used therein improved robustness, not all problems could be solved in the absence of limiters. 
In the broader context of HO methods, researchers typically achieve the stabilization required for industrial cases through strategies including limiting or filtering the solution or flux; using higher-strength numerical integration to better capture the nonlinear flux; or adding artificial dissipation to stabilize the solution. While such strategies demonstrate sufficient robustness to yield a solution, they may obscure important solution details, or even lead to an entropy violation. For this reason, it is of interest to develop a thorough understanding of the robustness of the FD-NSFR method in the absence of the aforementioned strategies. The investigation presented in this paper demonstrates that the FD-NSFR method may be more robust at large-time steps. 
FD-NSFR maintained accurate integrated variables and solution contours at a higher target time-step size than SD-NSFR.

In summary, in this work, we present a fully-discrete entropy stable flux reconstruction scheme in the broken Sobolev norm. The remainder of this paper will present our FD-NSFR scheme and numerical experiments that demonstrate its properties.
Following a brief discussion of notation in section \ref{sec:notation}, section~\ref{sec:RRK-overview} introduces the RRK approach of Ketcheson~\cite{ketcheson2019relaxation} and Ranocha~\cite{ranocha2020relaxation} in the context of an initial value problem which is conservative or dissipative for a general integrated numerical entropy.
The NSFR semidiscretization is presented in section \ref{sec:spatial-NSFR}. 
In section \ref{sec:scalar}, we present an FD-NSFR method for PDEs with an inner-product numerical entropy function, which is fully discretely stable in the broken Sobolev norm, along with the results of numerical experiments using the inviscid and viscous Burgers' equation.
We follow with a FD-NSFR method for PDEs that have a general convex numerical entropy function in section~\ref{sec:vector-NSFR}. Therein, we demonstrate that temporal numerical entropy is prevented in the $L^2$ norm, and present numerical experiments using the Euler and Navier-Stokes equations.
We close with an investigation into the robustness of the FD-NSFR scheme in section~\ref{sec:robustness}.

\section{Notation} \label{sec:notation}
We begin by introducing notation which will be used throughout this work. We use italic $a$ for scalar values, boldface $\mathbf{a}$ for vectors, and calligraphic $\mathcal{A}$ for matrices or matrix operators. Repeated indices indicate summation over the problem dimensions. Note that vectors are assumed to be {row} vectors, such that $\mathbf{a}^T$ is a column vector. As for subscripts, $h$ denotes the discrete solution, which is replaced by $m$ when the vector or matrix represents the solution or operator at a discrete element level. 
We use superscripts to indicate time indices, such that $\mathbf{u}^n=\mathbf{u}(t^n)$ is the solution at the $n$-th time step. 

\section{Relaxation Runge-Kutta Temporal Integration} \label{sec:RRK-overview}
Consider the initial value problem (IVP)
\begin{eqnarray}
\frac{\text{d} \mathbf{u}(t)}{\text{d} t} &=& \mathbf{f}(t,\mathbf{u}(t)), \quad t \in [t_o,t_f], \\
\mathbf{u}(t_o) &=& \mathbf{u}_o,
\label{ivp}
\end{eqnarray}
where $\mathbf{u} \colon \mathbb{R} \rightarrow \mathbb{R}^m$ and $\mathbf{f} \colon \mathbb{R} \times \mathbb{R}^m \rightarrow \mathbb{R}^m$. In this study, our emphasis is on IVPs that are either conservative or  dissipative for a specific inner-product norm: 
\begin{equation}
    \frac{\text{d} }{\text{d} t} \eta(\mathbf{u}(t)) = \langle \eta^\prime(\mathbf{u}(t)),\mathbf{f}(t,\mathbf{u}(t)) \rangle \leq 0,
\end{equation}
where $\eta$ is a smooth convex function. In the case of a scalar conservation law, $\eta(\mathbf{u}(t))$ can be taken to be $\frac{1}{2} \vert\vert \mathbf{u}(t) \vert\vert^2$ and is commonly referred to as the energy. Using the Euler and Navier-Stokes equations, $\eta$ represents the total entropy in the domain. For an energy or entropy-conservative system, 
\begin{displaymath}
    \eta (\mathbf{u}^{n+1}) = \eta (\mathbf{u}^n),
\end{displaymath}
while for a dissipative system,
\begin{displaymath}
    \eta (\mathbf{u}^{n+1}) \leq \eta (\mathbf{u}^n).
\end{displaymath}
For the remainder of this section,  we present the discretization of a time-independent ODE to simplify notation, though the methods can easily be extended to time-dependent ODEs. An RK method with $s$ stages advances the solution from the $n$-th time step to the $(n+1)$-th by constructing stages $\mathbf{u}^{(i)}$,
\begin{equation}
    \mathbf{u}^{(i)} = \mathbf{u}^{n} + \Delta t \sum_{j=0}^s a_{ij} \mathbf{f}(\mathbf{u}^{(i)}), \ \ \ \text{for}\  i = 1, ..., s,
    \label{eq:RK-stages}
\end{equation}
\noindent and assembling the stages together,
\begin{equation}
    \mathbf{u}^{n+1} = \mathbf{u}^n + \Delta t \sum_{i=0}^{s} b_i \mathbf{f}(\mathbf{u}^{(i)}),
    \label{eq:RK-assemble}
\end{equation}
\noindent where the coefficients $a_{ij}$ and $b_i$  of the Butcher tableau are chosen such that the method has a global truncation error of order $p$. We use explicit RK methods in this work, having $a_{ij}=0$ for $i \geq j$.

{\color{black}To ensure that the value of $\eta$ is conservative or dissipative over a time step,} 
we use the relaxation Runge-Kutta (RRK) method for the temporal discretization~\cite{ketcheson2019relaxation,ranocha2020relaxation}. 

Here, we consider that $\mathbf{u}$ is a global solution vector and $\mathbf{f}(\mathbf{u})$ is the unsteady residual resulting from a semi-discretely entropy stable spatial discretization.
The RRK method proceeds as a standard $s$-stage RK method to construct stages $\mathbf{u}^{(i)}$ according to Eq. (\ref{eq:RK-stages}), 
but adjusts the time step size by a relaxation parameter $\gamma^n$ at each time step,
\begin{equation}
    \mathbf{u}^{n+1}_\gamma = \mathbf{u}(t^n + \gamma^n \Delta t) = \mathbf{u}^n + \gamma^n \Delta t \sum_{i=0}^{s} b_i \mathbf{f}(\mathbf{u}^{(i)}),
\end{equation}
\noindent where $\mathbf{u}^{n+1}_\gamma \approx \mathbf{u}(t^n + \gamma^n \Delta t)$, with the relaxation parameter defined such that the entropy change is an estimate of the order of the RK scheme,
\begin{equation}
    0=\eta(\mathbf{u}^{n+1}_\gamma) - \eta(\mathbf{u}^n) - \gamma^n \Delta t \sum_{i=1}^s b_i \langle \mathbf{\eta}^\prime(\mathbf{u}^{(i)}), \mathbf{f}(\mathbf{u}^{(i)}) \rangle.
    \label{eq:rrk-root}
\end{equation}
\noindent The first two terms are the actual change in integrated numerical energy or entropy over the adjusted time step, while the third term estimates the change over the time step due to the spatial discretization.
Equation~(\ref{eq:rrk-root}) is solved either algebraically or numerically, depending on the form of the numerical entropy function. We choose to use the global RRK version in all cases, as the local version is not able to conserve entropy across the entire domain~\cite{ranocha2020fully}.

\section{Spatial Discretization: Semi-Discrete Nonlinearly Stable Flux Reconstruction} \label{sec:spatial-NSFR}
In this section, we will present the discretization of a multi-dimensional, vector-valued convection-diffusion equation of the general form,

\begin{equation}
\begin{cases}
    \frac{\partial}{\partial t}\mathbf{u}(\mathbf{x},t) +\nabla\cdot\mathbf{f}_c(\mathbf{u}(\mathbf{x},t) ) = \nabla\cdot\mathbf{f}_v(\mathbf{u}(\mathbf{x},t),\nabla \mathbf{u}(\mathbf{x},t) ),\quad \text{in } \Omega \times [0,t_f)\\
    \mathbf{u}(\mathbf{x},0)=\mathbf{u}_0(\mathbf{x}).
\end{cases}
\label{eq:3d_cons_law}
\end{equation}
Here, let $t_f >0$ be a final time and let $\Omega \subset \mathbb{R}^{d}$, where $d\leq 3$, be a bounded physical domain with Lipschitz boundary $\partial \Omega$. In addition,  $\mathbf{u}(\mathbf{x},t)$ denotes the state vector with $\mathbf{x}\coloneqq[x,y,z]\in\Omega$, $\mathbf{f}_c$ the convective flux vector, $\mathbf{f}_v$ is the viscous flux that is linear in the gradient $\nabla \mathbf{u}$, and $\mathbf{u}_0$ is an initial state.

\subsection{Preliminaries}
The discretization presented herein closely follows the formulation of Cicchino and co-authors~\cite{cicchino2022nonlinearly,cicchino2022provably,cicchino2023discretely,cicchino2024thesis}.  We will describe the formulation and features of the scheme, but direct the reader to the papers ~\cite{cicchino2022nonlinearly,cicchino2022provably,cicchino2023discretely} for complete details and proofs. {The methods described herein are implemented in the open-source Parallel High-Order Library for PDEs (\lstinline{PHiLiP}) developed by the Computational Aerodynamics Group at McGill University~\cite{shi2021full}.}

Let us consider that the physical domain $\Omega$ divided into $M$ non-overlapping elements $\Omega_m^d$, $\Omega \simeq \Omega^h \coloneqq \bigcup_{m=1}^M \Omega_m^d$, where, $m$ represents the $m$-th element. In this work we will only consider a discretization based on hexahedral tensor product elements.
We convert Eq. (\ref{eq:3d_cons_law}) to a system of first-order equations,
\begin{equation}
    \begin{cases}
        \frac{\partial}{\partial t}\mathbf{u}_h(\mathbf{x},t) +\nabla\cdot\mathbf{f}_c (\mathbf{u}_h(\mathbf{x},t) ) = \nabla\cdot\mathbf{f}_v(\mathbf{u}_h(\mathbf{x},t), \bm{\sigma}_h(\mathbf{x},t) ) \\
        \bm{\sigma}_h(\mathbf{u}(\mathbf{x},t) ) = \nabla \mathbf{u}_h(\mathbf{x},t),
    \end{cases}
    \label{eq:diffusion-sys}
\end{equation}
where $\mathbf{u}_h \in \mathbb{R}^{N_s}$ is the discontinuous discrete solution, $N_s$ is the number of components in the state, $\bm{\sigma}_h$ represents the corrected gradient of $\mathbf{u}_h$. The solution $\mathbf{u}_h$ is approximated and is assumed to belong to the DG approximation space of discontinuous piecewise polynomials 
\begin{equation}
    {\cal V}_h^p = \Bigl\{ v_h \in [L^2(\Omega)]^{N_s}, \quad v_h \vert_{\Omega_m^d} \in [{\cal P}^p (\Omega_m^d)]^{N_s} \;\;\forall \Omega_m^d \Bigr\},
\end{equation}
while $\bm{\sigma}_h$ belongs to 
\begin{equation}
    {\cal W}_h^p = \Bigl\{ \mathbf{w}_h \in [L^2(\Omega)]^{d\times N_s}, \quad \mathbf{w}_h \vert_{\Omega_m^d} \in [{\cal P}^p (\Omega_m^d)]^{d\times N_s} \;\;\forall \Omega_m^d \Bigr\},
\end{equation}
where $L^2(\Omega)$ is the space of square-integrable functions on the domain $\Omega$, and $p$ is the polynomial degree of the approximation space.

\noindent  We now present the discretization for a single state of the vector-valued $\mathbf{u}$, which we denote as $u$. The global approximate solution, ${u}_h(\mathbf{x},t)$ and the approximate auxiliary variable $\bm{\sigma}_h(\mathbf{x},t)$ of Eq.~(\ref{eq:diffusion-sys}) can be composed as piecewise polynomial approximations,
\begin{equation}
    {u}(\mathbf{x},t) \simeq {u}_h(\mathbf{x},t)=\bigoplus_{m=1}^M {u}_m(\mathbf{x},t), \ \ \ \ \ \  \bm{\sigma}(\mathbf{x},t) \simeq \bm{\sigma}_h(\mathbf{x},t)=\bigoplus_{m=1}^M \bm{\sigma}_m(\mathbf{x},t).
\end{equation}
\noindent On each element, $m$, we introduce the modal (hierarchical) polynomial basis functions of maximal order $p$ to represent the solution,
\begin{equation}
     {u}_m(\mathbf{x},t)  = \sum_{i=1}^{N_p}{{\chi}_{m,i}(\mathbf{x})\hat{{u}}_{m,i}(t)}, \ \ \ \ \ \  
     \bm{\sigma}_m(\mathbf{x},t)  = \sum_{i=1}^{N_p}{{\chi}_{m,i}(\mathbf{x})\hat{\bm{\sigma}}_{m,i}(t)},
\end{equation}
\noindent where $\mathbf{\hat{u}}_{m} \in \mathbb{R}^{N_{p}}$
is a row vector holding the modal coefficients on the $m$-th element for each state, and  
$\bm{\hat{\sigma}}_{m} \in \mathbb{R}^{d \times N_{p}}$ 
are the modal coefficients for the auxiliary variables. The vectors are of length $N_p=(p+1)^{d }$ for a $p$-th order polynomial basis.
The polynomial basis functions for the solution are defined as
\begin{equation}
     \bm{\chi}_m(\mathbf{x}) \coloneqq [\chi_{m,1}(\mathbf{x}), \ \chi_{m,2}(\mathbf{x}), \  \dots, \  \chi_{m,N_p}(\mathbf{x})]
     =\bm{\chi}(x)\otimes \bm{\chi}(y)\otimes \bm{\chi}(z).
\end{equation}
We exclusively use tensor product elements because they allow for sum-factorization techniques, enabling efficient scaling. The flux basis $\bm{\phi}_m$ is defined in the same way as the solution basis. 
In this work, for the interpolation or solution nodes, we choose Gauss-Lobatto-Legendre (GLL) nodes. 
We use either Gauss-Legendre (GL) nodes for integration, which exactly integrates polynomials up to degree $2p+1$, or GLL nodes, which exactly integrate polynomials up to degree $2p-1$. 
All numerical experiments use the same flux and solution {polynomial degree} without overintegrating the flux.

We choose to work in a transformed space, where each element is mapped to a standard reference element, $\Omega^r$. Coordinates are mapped from physical coordinates $\mathbf{x} = [x, y, z]$ to reference coordinates 
\begin{equation}
    \bm{\xi}^r \coloneqq \{ [\xi \text{, } \eta \text{, }\zeta]:-1\leq \xi,\eta,\zeta\leq1 \},
\end{equation}
\noindent using the mapping function $\bm{x}_m^c(\bm{\xi}^r)\coloneqq \bm{\Theta}_m(\bm{\xi}^r)$ and using the superscript $r$ to indicate quantities on the reference element. We use a reference element with volume nodes denoted as $\bm{\xi}_v^r$ and facet nodes as $\bm{\xi}_f^r$. The mapping introduces metrics terms and their evaluation as well as those of the transformed fluxes are critical to ensure free-stream preservation. The reader is advised to refer to~\cite{cicchino2022provably} for the proper treatment of the metric terms within an entropy/energy-conserving or stable scheme. {\color{black}All results in this work use straight-sided meshes, therefore we present most formulations without a detailed discussion of curvilinear coordinates.}

\subsection{Energy Stable Flux Reconstruction in Filtered DG Form}
Omitting the full derivation provided by Cicchino \textit{et al.}~\cite{cicchino2022nonlinearly,cicchino2022provably,cicchino2023discretely}, we write the discrete version of the strong formulation of the primary equation on a single element $m$ where we seek $\mathbf{u}_m \in {\cal V}_h^p$ such that for each test function $v_h \in {\cal V}_h^p$, 
\begin{equation}
\begin{aligned}
    {\mathcal{M}}_{m} \frac{d}{dt}\hat{\mathbf{u}}_m^T
    &+\bm{\chi}(\bm{\xi}_v^r)^T \mathcal{W} \nabla^r \bm{\phi}(\xi_v^r) \cdot
    \left( \hat{\mathbf{f}}^r_{c,m}-\hat{\mathbf{f}}^r_{v,m}
    \right)^T
    \\
    &+\sum_{f=1}^{N_f}\sum_{k=1}^{N_{fp}} \bm{\chi}(\bm{\xi}_{f,k}^r) \mathcal{W}_{f,k}
    \Bigl[
    \hat{\mathbf{n}}^r\cdot 
    \left(
    \left(\mathbf{f}_{c,m}^{*,r} - \mathbf{f}_{v,m}^{*,r} 
    \right)^T- \bm{\phi}(\xi_{f,k}^r) 
    \bigl( \hat{\mathbf{f}}^r_{c,m} -\hat{\mathbf{f}}^r_{v,m}
    \bigr)^T
    \right)\Bigr]
    =\mathbf{0}^T.
\end{aligned}
\label{eq:DGstrong-prim}
\end{equation}

We discretize the auxiliary equation in the same way, arriving at~\cite{cicchino2024thesis}
\begin{equation}
    \begin{aligned}
    {\mathcal{M}}_{m} \hat{\bm{\sigma}}_m ^T
    &- 
    \bm{\chi}(\bm{\xi}_v^r)^T \mathcal{W} \nabla^r \bm{\phi}(\xi_v^r) (\hat{\mathbf{u}}^r_m)^T \\
    &-\sum_{f=1}^{N_f}\sum_{k=1}^{N_{fp}} \bm{\chi}(\bm{\xi}_{f,k}^r)\mathcal{W}_{f,k}
    \Bigl[\hat{\mathbf{n}}^r \Bigl( (\mathbf{u}_m^{*,r})^T- \bm{\phi}(\xi_{f,k}^r) (\hat{\mathbf{u}}^r_{m})^T\Bigr)\Bigr]=\mathbf{0}^T.
\end{aligned}
    \label{eq:DGstrong-aux}
\end{equation}
In the two preceding equations, the superscript $^*$ denotes a numerical flux and $\hat{\mathbf{n}}^r$ is a unit normal vector in the reference element.
The convective numerical flux ${\mathbf{f}}_{c,m}^{*,r}$, solution numerical flux $\mathbf{u}^{*,r}_m$ and diffusive numerical flux $\mathbf{f}_{v,m}^{*,r}$ are chosen according to desired stability properties.
The quadrature weights are stored in diagonal operators $\mathcal{W}$ at volume nodes and $\mathcal{W}_f$ at face nodes. 
The Jacobian $\mathcal{J}_m$ is evaluated at volume nodes. The mass matrix $ \mathcal{M}_m \in \mathbb{R}^{N_p} \times \mathbb{R}^{N_p}$ is evaluated using volume nodes,
\begin{equation}
    \begin{aligned}
       \mathcal{M}_m= \bm{\chi}(\bm{\xi}_v^r)^T\mathcal{W}\mathcal{J}_m\bm{\chi}(\bm{\xi}_v^r).
    \end{aligned}\label{eq: mass}
\end{equation}
ESFR is a cousin of DG that adds a correction function to the flux such that it is $C^0$ continuous between elements~\cite{huynh_flux_2007}. 
The divergence of the correction functionis in the same polynomial space as the solution. We implement ESFR as a modally-filtered DG method in the approach of Zwanenburg and Nadarajah~\cite{zwanenburg2016equivalence}, such that the only difference between strong DG as it was expressed in Eq.~(\ref{eq:DGstrong-prim}-\ref{eq:DGstrong-aux}) and ESFR is a modification to the mass matrix indicated by the tilde,
\begin{equation}
\begin{aligned}
    \tilde{\mathcal{M}}_{m,p} \frac{d}{dt}\hat{\mathbf{u}}_m^T
    &+\bm{\chi}(\bm{\xi}_v^r)^T \mathcal{W} \nabla^r \bm{\phi}(\xi_v^r) \cdot
    \left( \hat{\mathbf{f}}^r_{c,m}-\hat{\mathbf{f}}^r_{v,m}
    \right)^T
    \\
    &+\sum_{f=1}^{N_f}\sum_{k=1}^{N_{fp}} \bm{\chi}(\bm{\xi}_{f,k}^r) \mathcal{W}_{f,k}
    \Bigl[
    \hat{\mathbf{n}}^r\cdot 
    \left(
    \left(\mathbf{f}_{c,m}^{*,r} - \mathbf{f}_{v,m}^{*,r} 
    \right)^T 
    - \bm{\phi}(\xi_{f,k}^r) 
    \bigl( \hat{\mathbf{f}}^r_{c,m} -\hat{\mathbf{f}}^r_{v,m}
    \bigr)^T
    \right)\Bigr]
    =\mathbf{0}^T \\
    \tilde{\mathcal{M}}_{m,a} \hat{\bm{\sigma}}_m ^T
    &- 
    \bm{\chi}(\bm{\xi}_v^r)^T \mathcal{W} \nabla^r \bm{\phi}(\xi_v^r) (\hat{\mathbf{u}}^r_m)^T \\
    &-\sum_{f=1}^{N_f}\sum_{k=1}^{N_{fp}} \bm{\chi}(\bm{\xi}_{f,k}^r)\mathcal{W}_{f,k}
    \Bigl[\hat{\mathbf{n}}^r ( (\mathbf{u}_m^{*,r})^T- \bm{\phi}(\xi_{f,k}^r) (\hat{\mathbf{u}}^r_{m})^T)\Bigr]=\mathbf{0}^T.
\end{aligned}
    \label{eq:ESFR}
\end{equation}
ESFR was introduced as filtered DG in a one-dimensional linear advection case by Allaneau and Jameson~\cite{allaneau2011connections}, and generalized to three-dimensional, nonlinear equations by Zwanenburg and Nadarajah~\cite{zwanenburg2016equivalence}. In this work, we use the same concept for the auxiliary formulation.
The modified mass matrix for the primary and auxiliary equations is a sum of the DG mass matrix and an ESFR contribution matrix,
\begin{equation}
    \tilde{\mathcal{M}}_{m,p} = \mathcal{M}_m + \mathcal{K}_p, \ \ \ \  \tilde{\mathcal{M}}_{m,a} = \mathcal{M}_m + \mathcal{K}_a,
\end{equation}
\noindent 

For clarity, we define the FR contribution operators for one-, two- and three-dimensional domains separately. When $d=1$, the FR contribution operators for a discretization of polynomial degree $p$ are
\begin{equation}
    \mathcal{K}_{p,1D}
   = c (\mathcal{D}_\xi^p)^T \mathcal{M}_m \mathcal{D}^p, \ \ \ \ \ \ 
   \mathcal{K}_{a,1D}
   = \kappa (\mathcal{D}_\xi^p)^T \mathcal{M}_m \mathcal{D}_\xi^p. 
\label{defineKmatrix}
\end{equation}
using $c$ in a normalized Legendre reference basis, as discussed in~\cite[Remark 2.1]{cicchino2022nonlinearly}.

$\mathcal{D}_\xi^p$ is the modal differential operator from the unfiltered strong DG formulation raised to the power $p$,
\begin{equation}
    \mathcal{D}_\xi^p=(\mathcal{M}^{-1} \bm{\chi}(\bm{\xi}_v^r)^T\mathcal{W} \frac{\partial \bm{\chi}}{\partial \xi}(\bm{\xi}_v^r) )^p.
    \label{eq:diff_oper}
\end{equation}

When $d=2$, the FD correction operators are,
\begin{equation}
\begin{aligned}
    \mathcal{K}_{p,2D}
   &=\sum_{s,v } c_{(s,v)}\Big(\mathcal{D}_\xi^s \mathcal{D}_\eta^v \Big)^T\mathcal{M}_m\Big(\mathcal{D}_\xi^s \mathcal{D}_\eta^v \Big), \\
   \mathcal{K}_{a,2D}
   &=\sum_{s,v } \kappa_{(s,v)}\Big(\mathcal{D}_\xi^s \mathcal{D}_\eta^v\Big)^T\mathcal{M}_m\Big(\mathcal{D}_\xi^s \mathcal{D}_\eta^v \Big),
    \label{eq:Km2D}
\end{aligned}
\end{equation}
where $c_{(s,v)}$ and $\kappa_{(s,v)}$ are found according to the spatial polynomial order $p$ and the 1D correction parameter $c$ or $\kappa$~\cite{cicchino2022provably},
\begin{equation}
    c_{(s,v)} = c^{(s/p + v/p)}, \ \ \ \ \kappa_{(s,v)} = \kappa^{(s/p + v/p)},
\end{equation}
where $s$ and $v$ are integer sets $\{0,p\}$ with the condition $s+v\geq p$. The modal differential operator for the second reference coordinate $\mathcal{D}_\eta$ is defined alike Eq. (\ref{eq:diff_oper}).

Extending to $d=3$,
\begin{equation}
\begin{aligned}
    \mathcal{K}_{p,3D}
   &=\sum_{s,v,w } c_{(s,v,w)}\Big(\mathcal{D}_\xi^s \mathcal{D}_\eta^v\mathcal{D}_\zeta^w \Big)^T\mathcal{M}_m\Big(\mathcal{D}_\xi^s \mathcal{D}_\eta^v\mathcal{D}_\zeta^w \Big), \\
   \mathcal{K}_{a,3D}
   &=\sum_{s,v,w } \kappa_{(s,v,w)}\Big(\mathcal{D}_\xi^s \mathcal{D}_\eta^v\mathcal{D}_\zeta^w \Big)^T\mathcal{M}_m\Big(\mathcal{D}_\xi^s \mathcal{D}_\eta^v\mathcal{D}_\zeta^w \Big),
    \label{eq:Km3D}
\end{aligned}
\end{equation}
where $c_{(s,v,w)}$ and $\kappa_{(s,v,w)}$ are defined similarly to the $d=2$ case~\cite{cicchino2022provably},
\begin{equation}
    c_{(s,v,w)} = c^{(s/p + v/p+ w/p)}, \ \ \ \ \kappa_{(s,v,w)} = \kappa^{(s/p + v/p+ w/p)},
\end{equation}
and the modal differential operator for the final reference coordinate $\mathcal{D}_\zeta$ is also defined alike Eq. (\ref{eq:diff_oper}).

Various ESFR schemes can be recovered through the choice of $c$, for instance, those defined by Castonguay \textit{et al.}~\cite{castonguay2012new}. 
As the value of $c$ increases, it is expected that the linear stability restriction will allow larger time steps, until a point at which an order is lost~\cite{castonguay2012new}.
While $c$ and $\kappa$ filter the primary and auxiliary equations separately~\cite{castonguay2013energy}, we use the same values $c=\kappa$ to generate results in this work.

We use three FR schemes to generate numerical results in this paper. They are listed in order of increasing magnitude: $c_{DG}=0$, which recovers DG; $c_{Hu}$, which recovers Huynh's~\cite{huynh_flux_2007} original FR scheme~\cite{vincent_new_2011}; and $c_{+}$, which was found to be the largest $c$ value that preserves the convergence order for linear advection~\cite{castonguay2012high}. 
\subsection{Nonlinearly Stable Flux Reconstruction}

The entropy-stable split form discretizations used herein  are developed by Cicchino, Nadarajah and co-authors~\cite{cicchino2022nonlinearly,cicchino2022provably,cicchino2023discretely}.
The most recent paper~\cite{cicchino2023discretely} extends to three-dimensional vector-valued PDEs, proving entropy stability, global conservation, and free-stream preservation. 
We use the same primary-auxiliary set as in of~\cite[Eq. (6.8)]{cicchino2024thesis} for the final NSFR semidiscretization. The final computational form for the primary equation is
\begin{equation}
\begin{aligned}
       \frac{d}{d t} \hat{\mathbf{u}}_m^T = (\tilde{\mathcal{M}}_{m,p})^{-1} \Bigl(
     &- \left[ \bm{\chi}(\bm{\xi}_v^r)^T\:  \bm{\chi}(\bm{\xi}_f^r)^T\right]
     \left[ \left(\Tilde{\mathcal{Q}}-\Tilde{\mathcal{Q}}^T\right) \odot \mathcal{F}_{c,m}^r \right]\mathbf{1}^T \\
     &- \sum_{f=1}^{N_f}\sum_{k=1}^{N_{fp}}
    \bm{\chi}(\bm{\xi}_{fk}^r)^T \mathcal{W}_{fk} \hat{\mathbf{n}}^r \cdot
    (\mathbf{f}_{c,m}^{*,r} )^T\\
    &+\bm{\chi}(\bm{\xi}_v^r)^T \mathcal{W} \nabla^r \bm{\phi}(\xi_v^r) \cdot (\hat{\mathbf{f}}^r_{m,v})^T \\
    &+\sum_{f=1}^{N_f}\sum_{k=1}^{N_{fp}} \bm{\chi}(\bm{\xi}_{f,k}^r){W}_{f,k}
    \Bigl[
    \hat{\mathbf{n}}^r\cdot 
    \bigl(\mathbf{f}_{v,m}^{*,r} - \bm{\phi}(\xi_{f,k}^r) \hat{\mathbf{f}}^r_{v,m} 
    \bigr)^T\Bigr] 
    \Bigr),
\end{aligned}
    \label{eq:NSFR-vector}
\end{equation}
\noindent where the general skew-symmetric stiffness operator of Chan~\cite{chan2018discretely} is employed for the convective part,
\begin{equation}
    \Tilde{\mathcal{Q}}-\Tilde{\mathcal{Q}}^T = 
    \begin{bmatrix}
        \mathcal{W}\nabla^r\bm{\phi}(\bm{\xi}_v^r) - \nabla^r\bm{\phi}(\bm{\xi}_v^r) ^T \mathcal{W}
        &  \sum_{f=1}^{N_f} \bm{\phi}(\bm{\xi}_f^r)^T \mathcal{W}_f \text{diag}(\hat{\mathbf{n}}^r_f) \\
    -  \sum_{f=1}^{N_f}  \mathcal{W}_f \text{diag}(\hat{\mathbf{n}}^r_f) \bm{\phi}(\bm{\xi}_f^r) & \mathbf{0} 
    \end{bmatrix}
    \label{eq: skew-symm qtilde oper}
\end{equation}
\noindent and the two-point convective reference fluxes are stored in $\left( \mathcal{F}^r_{c,m}\right)_{ij}$, defined as 
\begin{equation}
     {  \left( \mathcal{F}^r_{c,m}\right)_{ij}} = 
     {\mathbf{f}_{c}\left(\Tilde{\mathbf{u}}_m(\bm{\xi}_{i}^r),\Tilde{\mathbf{u}}_m(\bm{\xi}_{j}^r)\right)}
     {\left(
       \frac{1}{2}\left(\mathbf{C}_m(\bm{\xi}_{i}^r)+ \mathbf{C}_m(\bm{\xi}_{j}^r) \right)\right)},\:\forall\: 1\leq i,j\leq N_v + N_{fp}.
       \label{eq:two-pt-flux-matrix}
\end{equation}
\noindent The symbol $\odot$ indicates a Hadamard product, which is evaluated using sum factorization~\cite{cicchino2024scalable}. The metric Jacobian cofactor matrix $\mathbf{C}_m$ is formulated as described in~\cite[Section 3.2]{cicchino2022provably}, following the invariant curl form of Kopriva~\cite{kopriva2006metric}, to discretely uphold the geometric conservation law for a fixed mesh. Metric dependence only impacts curvilinear meshes, while we present results only on straight meshes.
Crucially, when we discretize multi-dimensional, vector-valued PDEs, the two-point convective flux in Eq.~(\ref{eq:two-pt-flux-matrix}) is evaluated using entropy-projected variables $\Tilde{\mathbf{u}}_m$, per~\cite[Eq. (42)]{cicchino2023discretely}. This procedure, which was first proposed by Chan~\cite{chan2018discretely}, is key in enabling entropy stability on uncollocated solution and flux or integration nodes. 

The NSFR discretization applies splitting to both the volume and face terms. 
Additionally, the ESFR correction impacts all split terms, both face and volume, due to its construction in the modified mass matrix.
We pair the split form in Eq.~(\ref{eq:NSFR-vector}) with the same auxiliary equation as the second line of Eq.~(\ref{eq:ESFR}) to discretize the viscous system.
The method has been demonstrated with Burgers' equation~\cite{cicchino2022nonlinearly}, the 3D Euler equations~\cite{cicchino2023discretely}, and the 3D Navier-Stokes equations~\cite{cicchino2024thesis,brillon2023use}.

\begin{remark} \label{rem:entropy-stab-spatial} As stated in~\cite[Theorem 6.3]{cicchino2023discretely} the NSFR discretization is discretely entropy conserving if the two-point flux satisfies the Tadmor shuffle condition and entropy stable if additional dissipation through upwinding is present. {\color{black} Define $\mathbf{v} = S'(\mathbf{u})$ as the entropy variables for the convex function, $S(\mathbf{u})$, 
with the total integrated entropy over the computational domain defined as $\eta(\mathbf{u}) = \int_{\Omega} S(\mathbf{u})\; d\Omega$. If we represent $\mathbf{v}$} in the same polynomial space as that of the state $\mathbf{u}$, as $\mathbf{v} = \bm{\chi}(\bm{\xi}_v^r) \mathbf{\hat v}^T$, where $\mathbf{v} \in {\cal V}_h^p$, then the NSFR scheme can be shown to be entropy stable in the discrete broken Sobolev norm based on~\cite[Section 6.3]{cicchino2022provably},
\begin{equation}
    \hat{\mathbf{v}}^T \tilde{\mathcal{M}} \frac{d}{dt} \hat{\mathbf{u}} \leq 0.
    \label{eq:discrete-broken-sobolev}
\end{equation}
\end{remark}

The results of the previous papers~\cite{cicchino2022nonlinearly,cicchino2022provably,cicchino2023discretely} present the discrete numerical entropy derivative in Eq. (\ref{eq:discrete-broken-sobolev}), and confirm that it is conserved when expected by the problem's physics.
However, the entropy stability proofs are semi-discrete, and do not address numerical entropy generated by the temporal discretization.

In addition to provable stability properties, the semidiscretization demonstrates scaling at $\mathcal{O}(p^{d+1})$~\cite{cicchino2023discretely}. This is enabled by sum-factorization. 
In particular, efficient evaluation of the Hadamard product which appears in the split form is described in a technical note~\cite{cicchino2024scalable}. 
\subsubsection{Choice of numerical flux functions} \label{sec:num-flux-choice}

The choice of numerical fluxes for the convective numerical flux, $\mathbf{f}_{c}^*$ is expanded upon after the introduction of each of the convection-diffusion governing equations considered in Sections~\ref{sec:scalar} and~\ref{sec:NS-euler-governing-eqns}.

Upwinding can be used alongside the convective numerical flux in Eq. (\ref{eq:two-pt-flux-matrix}) to formulate an entropy-stable scheme. We use local Lax-Friedrichs upwinding, such that an entropy-stable flux $\mathbf{f}_{c,ES}$ will be formulated by comparison with an entropy-conserving two-point flux, $\mathbf{f}_{c,EC}$,
\begin{equation}
    \mathbf{f}_{c,ES} = \mathbf{f}_{c,EC} - \frac{1}{2} \vert \lambda_{\text{max}} \vert [\![\mathbf{u}]\!],
\end{equation}
\noindent where $\lambda_{\text{max}}$ is the maximum convective eigenvalue at the interface.
We  use DG notation, with the jump and average value for vector-valued quantities respectively defined as
\begin{equation}
    [\![\mathbf{u}]\!] = \hat{\mathbf{n}}^- \cdot \mathbf{u}^- + \hat{\mathbf{n}}^+ \cdot \mathbf{u}^+, \ \ \ \ \ \{\!\!\{\mathbf{u}\}\!\!\} = \frac{1}{2} (\mathbf{u}^- + \mathbf{u}^+),
\end{equation}
\noindent across a face which has interior state $\mathbf{u}^-$ and exterior state $\mathbf{u}^+$.

As for the viscous numerical flux and solution numerical flux, various formulations have been shown to be stable for the linear diffusion equation~\cite{castonguay2013energy,quaegebeur2019stability,quaegebeur2019stability2D}. For the Navier-Stokes equations, we choose the symmetric interior penalty approach~\cite{arnold1982interior}, defining the solution numerical flux and viscous numerical flux as
\begin{equation}
    \mathbf{u}^* = \{\!\!\{\mathbf{u}\}\!\!\}, \ \ \ \  \mathbf{f}_v^* = \{\!\!\{\mathbf{f}_v(\mathbf{u}, \bm{\sigma})\}\!\!\} - \tau [\![\mathbf{u}]\!],
    \label{eq:SIP}
\end{equation}
\noindent where $\tau$ is the penalty term. 

On the other hand, when we discretize the 1D viscous Burgers' equation, we wish to add no extra dissipation to the fluxes. Therefore, we use a central viscous approach,
\begin{equation}
    {u}^* = \{\!\!\{{u}\}\!\!\}, \ \ \ \  {f_v}^* = \{\!\!\{f_v(u,\sigma)\}\!\!\}.
    \label{eq:central-viscous}
\end{equation}

\section{Fully-Discrete Entropy-Stable Schemes for the Scalar Convection-Diffusion Equation} \label{sec:scalar}

{In this section, we present time-stepping strategies for the nonlinearly-stable flux reconstruction scheme when numerical entropy can be expressed in inner-product form. The aim is to enforce the fully discrete entropy inequality through the algebraic RRK approach~\cite{ketcheson2019relaxation}, which we denote as FD-NSFR, while a pairing of standard RK with the NSFR is deemed ``semi-discrete NSFR" (SD-NSFR). We first extend RRK for {\color{black} scalar-valued} NSFR, then verify the implementation for the Burgers' inviscid and viscous equations.}

\subsection{Relaxation Runge-Kutta for the NSFR Scheme} \label{sec:RRK-implementation-scalarc^{k,2}}
Before we proceed with the use of the relaxation Runge-Kutta method for the NSFR scheme, we define the employed spaces. 

\begin{definition}
The classical Sobolev space is defined as 
$$W^{k,p}(\Omega) := \{ \mathbf{u} \in [L^2(\Omega)]^{d\times N_s} : \partial^{\alpha} \mathbf{u} \in [L^2(\Omega)]^{d\times N_s}, |\alpha| \leq k \}$$
with the norm 
    \begin{displaymath}
        \vert\vert \mathbf{u} \vert\vert_{W^{k,p}(\Omega)}^p = \sum_{|\alpha| \leq k}\int_{\Omega} \left(\partial^{\alpha} \mathbf{u} \right)^p\; d\xi,
    \end{displaymath}
and the corresponding seminorm of order $k$
    \begin{equation}
        \vert \mathbf{u} \vert_{W^{k,p}(\Omega)}^p = \sum_{|\alpha| = k}\int_{\Omega} \left(\partial^{\alpha} \mathbf{u} \right)^p\; d\xi.
        \label{eq:k-seminorm}
    \end{equation}

\end{definition}

\begin{definition}
In the case $p=2$, the Sobolev space forms a Hilbert space, with $W^{k,2} = H^k$, where   the Hilbert space $H^k$ admits an inner product and is defined in terms of the $L^2$ inner product
\begin{displaymath}
    \langle \mathbf{u}, \mathbf{v} \rangle_{W^{k,2}(\Omega)} = \sum_{|\alpha| \leq k} \langle \partial^{\alpha} \mathbf{u}, \partial^{\alpha} \mathbf{v} \rangle_{L^2(\Omega)}.
\end{displaymath}
\end{definition}

\begin{definition}
    We define a broken Sobolev space as one that contains only the weighted sum of the first and last in the sequence
        \begin{displaymath}
        \vert\vert \mathbf{u} \vert\vert_{W^{k,p}_c(\Omega)}^p = \int_{\Omega} \mathbf{u}^p + c\left(\partial^{\alpha} \mathbf{u} \right)^p\; d\xi,
    \end{displaymath}
    with the parameter $c$ as the weight.
    \label{defineBrokenSobolevNorm}
\end{definition}

\begin{definition}
    For both linear and nonlinear conservation laws, the semi-discrete flux reconstruction approach as presented in the previous section establishes energy or nonlinear/entropy stability in the stated broken Sobolev space $W^{k,2}_c(\Omega)$ with inner product $\langle \cdot, \cdot \rangle$, inducing the norm, 
    \begin{displaymath}
        \frac{\text{d}}{\text{d}t} \vert\vert \mathbf{u} \vert\vert_{W^{k,2}_c(\Omega)}^2 = \frac{\text{d}}{\text{d}t} \int_{\Omega} \mathbf{u}^2\ +c\left(\frac{\partial^k \mathbf{u}}{\partial \xi^k} \right)^2\; d\xi \leq 0,
    \end{displaymath}
    where the modified norm only contains the zeroth and $p$-th derivative terms. The subscript $c$ in $W^{k,2}_c(\Omega)$ signifies the use of the broken Sobolev space. {\color{black} For brevity, we will write the broken Sobolev space as $W^{k,2}_c$ in subsequent sections.}
\end{definition}

Thus the objective of this work is to extend the semidiscrete energy stability to a fully-discrete scheme through the use of the relaxation Runge-Kutta approach. For initial value problems~(\ref{ivp}) having energy, $\eta := \frac{1}{2} \langle \mathbf{u},\mathbf{u} \rangle_{W^{k,2}_c}$, as the numerical entropy variable, we desire the discrete solution of the NSFR scheme to be \textit{monotonicity preserving}, $\vert\vert \mathbf{u}^{n+1} \vert\vert_{W^{k,2}_c} \leq \vert\vert \mathbf{u}^{n} \vert\vert_{W^{k,2}_c}$ for dissipative systems and discretely conserve energy $\vert\vert \mathbf{u}^{n+1} \vert\vert_{W^{k,2}_c} = \vert\vert \mathbf{u}^{n} \vert\vert_{W^{k,2}_c}$ for conservative systems. 

\noindent Thus, the change of energy from one-time step to the next can be expanded from equation~\ref{eq:rrk-root},
\begin{eqnarray}
    \vert\vert\mathbf{u}^{n+1}_\gamma \vert\vert^2_{W^{k,2}_c} - \vert\vert\mathbf{u}^n \vert\vert^2_{W^{k,2}_c}= && 2\gamma^n \Delta t  \sum_{i=1}^s b_i \langle \mathbf{u}^{(i)}, \mathbf{f}(\mathbf{u}^{(i)}) \rangle_{W^{k,2}_c}\nonumber \\
    &-& 2\gamma^n \Delta t^2  \sum_{i,j=1}^s b_i a_{ij}\langle \mathbf{f}(\mathbf{u}^{(j)}), \mathbf{f}(\mathbf{u}^{(i)}) \rangle_{W^{k,2}_c}\nonumber \\
    &+& (\gamma^n)^2 \Delta t^2  \sum_{i,j=1}^s b_i b_{j}\langle \mathbf{f}(\mathbf{u}^{(j)}), \mathbf{f}(\mathbf{u}^{(i)}) \rangle_{W^{k,2}_c},
    \label{eq:rrk-root-expanded}
\end{eqnarray}
where we evaluate the inner products in the appropriate broken Sobolev norm per Definition~\ref{defineBrokenSobolevNorm},
\begin{equation}
    \langle \mathbf{a} , \mathbf{b} \rangle_{W^{k,2}_c} := \mathbf{a} ^T (\mathcal{M} + \mathcal{K}) \mathbf{b},
\end{equation}
where $\mathcal{K}$ is defined as equation~\ref{defineKmatrix}; since nonlinear stability for the NSFR scheme is established in the stated norm. We can thus find the relaxation parameter by solving an algebraic equation at each time step to eliminate the final two terms similar to ~\cite{ketcheson2019relaxation} and ensure that there is no temporal entropy production,
\begin{equation}
\gamma^n = 
\begin{cases}
    \frac{2\sum_{i,j=1}^m b_i a_{ij} \langle  \mathbf{f}(\mathbf{u}^{(j)}), \mathbf{f}(\mathbf{u}^{(i)})\rangle_{W^{k,2}_c} }{\sum_{i,j=1}^{m}b_i b_j\langle  \mathbf{f}(\mathbf{u}^{(i)}), \mathbf{f}(\mathbf{u}^{(j)})\rangle_{W^{k,2}_c} }, \ \ \ \text{if}\ \sum_{i,j=1}^{m}b_i b_j\langle  \mathbf{f}(\mathbf{u}^{(i)}), \mathbf{f}(\mathbf{u}^{(j)})\rangle_{W^{k,2}_c}  \neq 0,  \\
    1, \ \ \ \ \ \ \ \ \ \ \ \ \ \ \ \ \ \ \ \ \ \ \ \ \ \ \ \ \ \text{if} \  \sum_{i,j=1}^{m}b_i b_j\langle  \mathbf{f}(\mathbf{u}^{(i)}), \mathbf{f}(\mathbf{u}^{(j)})\rangle_{W^{k,2}_c}  = 0.
\end{cases}
\label{eq:RRK_explicit}
\end{equation}

\begin{remark}
    In regards to the existence of a solution, as stated in~\cite[Lemma 2.1]{ketcheson2019relaxation}, regardless of the chosen norm to conserve energy, if $\sum_{i,j=1}^m b_i a_{ij} \ge 0$ and $\gamma^n$ is defined as~(\ref{eq:RRK_explicit}), then $\gamma^n  \ge 0$ for sufficiently small $\Delta t \ge 0$.
    \label{remark_on_existence_RRK}
\end{remark}

\begin{remark}
    Since the Runge-Kutta coefficients $(a,b)$ are nonnegative, then the relaxation method is \textit{monotonicity} preserving if $\Delta t$ is chosen such that $\gamma^n \ge 0$ from (\ref{eq:RRK_explicit}) with inner-products computed in the broken Sobolev norm.
\end{remark}

We restate a crucial lemma from Ketcheson~\cite[Lemma 2.8]{ketcheson2019relaxation} for completeness. Readers are advised to refer to the referenced article for a complete proof.  
\begin{lemma}~\cite[Lemma 2.8]{ketcheson2019relaxation}
    Let $a_{ij}, b_j$ denote the coefficients of an RK method of order $p$, let $f$ be a sufficiently smooth function, and let $\gamma^n$ be defined as~(\ref{eq:RRK_explicit}) and satisfy conditions of Remark~\ref{remark_on_existence_RRK}. Then
    \begin{displaymath}
        \gamma^n = 1 + {\cal O}(\Delta t^{p-1}).
    \end{displaymath}
    \label{lemma_fdnsfr_accuracy}
\end{lemma}

We will illustrate Lemma~\ref{lemma_fdnsfr_accuracy} in the following subsection, where we will show the convergence of $\gamma^n$ towards 1 at the expected order.

\subsection{Results using the inviscid Burgers' equation}\label{sec:BurgersInviscid}

We first use the 1D inviscid Burgers' equation in order to evaluate the convergence behavior of the algebraic version of RRK in time step refinement studies. 
We define the inviscid Burgers' equation using the general convection-diffusion equation (\ref{eq:3d_cons_law}):
\begin{equation}
    \begin{aligned}
        &f_c(u) = \frac{u^2}{2}, \ \ f_v(u, \nabla u) = 0\\
        &u_0(x) = \sin(\pi x) \\
        &x \text{ periodic on } [0,2],\ \ \ t \in [0,0.3],
    \end{aligned}
    \label{eq:burg-inv}
\end{equation}
\noindent which has an inner-product numerical entropy function $\eta := \frac{1}{2} \langle u,u \rangle_{W^{k,2}_c}$. 
We use the following choice of energy-conserving numerical flux~\cite{tadmor2003entropy},
\begin{equation}
    f^*_c = \frac{1}{6} \left(
        u^- u^- + u^- u^+ + u^+ u^+
    \right).
    \label{eq:burgers-flux}
\end{equation}
\noindent where $u^-$ and $u^+$ are the solution states on the interior and exterior sides of an element face. This choice of numerical flux conserves energy when suitable split forms are used; see~\cite[section 4.2]{gassner2013skew} or~\cite[section 3.1.2]{cicchino2022nonlinearly}. 
The numerical flux does not add any dissipation in the form of upwinding.

We use a fine grid with $32$ elements and $p=4$, such that temporal error dominates.
The tests use equally-spaced elements with collocated Gauss-Lobatto-Legendre (GLL) solution and flux nodes. 
We use three explicit RK methods, abbreviated as follows:
\begin{itemize}
    \item RK2: Heun's method, a 2-stage, 2nd-order explicit method, which is also known in literature as second-order strong-stability preserving,
    \item SSPRK3: The strong-stability preserving 3rd-order, 3-stage method of Shu and Osher~\cite{shu1988efficient}, and
    \item RK4: The classical fourth-order, four-stage method.
\end{itemize}

We use the algebraic form of RRK~\cite{ketcheson2019relaxation} as described in section \ref{sec:RRK-implementation-scalarc^{k,2}}. 
We evaluate convergence from a large time step size, performing six refinements by a factor of two.
{\color{black} For the inviscid Burgers' test case, we defined the largest time step size as the largest even divisor of the end time that ensures asymptotic convergence behavior.}
In Fig.~\ref{fig:Burgers-convergence}(left) we compare the $L^2$ norm of the solution to a reference solution using the same spatial discretization and a very small time step.
We observe that both the semi-discrete and fully-discrete NSFR schemes yield similar error convergence behavior and we obtain the expected order for each of the temporal integration schemes. 
The average relaxation parameter shown in Fig.~\ref{fig:Burgers-convergence}(center) converges at the expected order based on Lemma~\ref{lemma_fdnsfr_accuracy}: both the SSPRK3 and RK4 methods converge at $p-1$; while, according to Ketcheson~\cite{ketcheson2019relaxation}, symmetry causes the RK2 method to converge at $p$.
The rightmost plot in Fig.~\ref{fig:Burgers-convergence} verifies the entropy-conserving implementation of the FD-NSFR scheme. 
The FD-NSFR schemes conserve energy to machine precision, while the SD-NSFR schemes have a change in energy, which converges at the expected order.

\begin{figure}[h!]
    \centering
    \includegraphics[height=2.5 in]{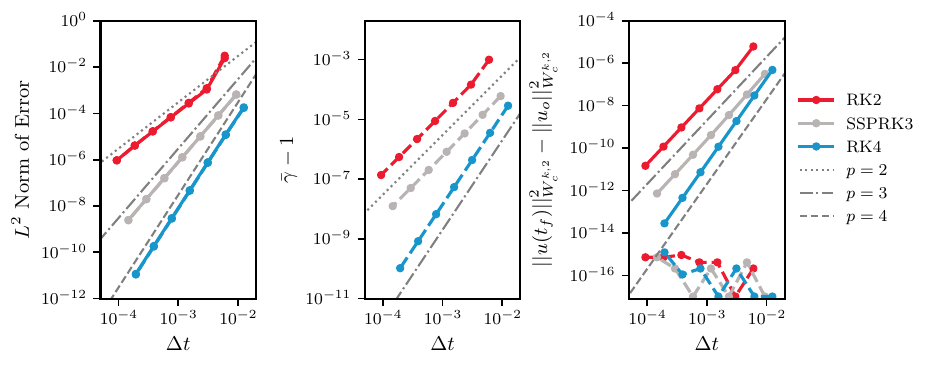}
    \caption{Convergence plots for the inviscid Burgers' test case using $c_{DG}$. The solid lines use SD-NSFR, and the dashed lines use FD-NSFR. Left: convergence of $L^2$ solution error at the end time compared to a calculation using a very  small time step. Center: convergence of the average relaxation parameter to $1$ in the FD-NSFR cases. Right: energy change in the broken Sobolev $W^{k,2}_c$ norm at the end time relative to the initial condition. Machine zero is shown as $10^{-17}$ such that it can be plotted on a logarithmic scale.}
    \label{fig:Burgers-convergence}
\end{figure}

Next, we use the 1D inviscid Burgers' test case to verify the FD-NSFR implementation when changing the correction parameter. 
Figure \ref{fig:Burgers_FR} demonstrate similar convergence behaviour between all FR variations. While results are presented only for SSPRK3, a similar trend can be seen for other RK schemes.
{\color{black}The primary benefit of FR is the ability to increase the time step size~\cite{huynh_flux_2007}, as the value of $c$ (Equation~\ref{defineKmatrix}) is increased while maintaining the order of the spatial discretization.
When $c_{DG}$ is used, at least $32$ time steps must be taken in order to converge asymptotically. That figure is $30$ and $29$ time steps for $c_{Hu}$ and $c_+$, respectively.}

\begin{figure}[h!]
    \centering
    \includegraphics[height=2.5 in]{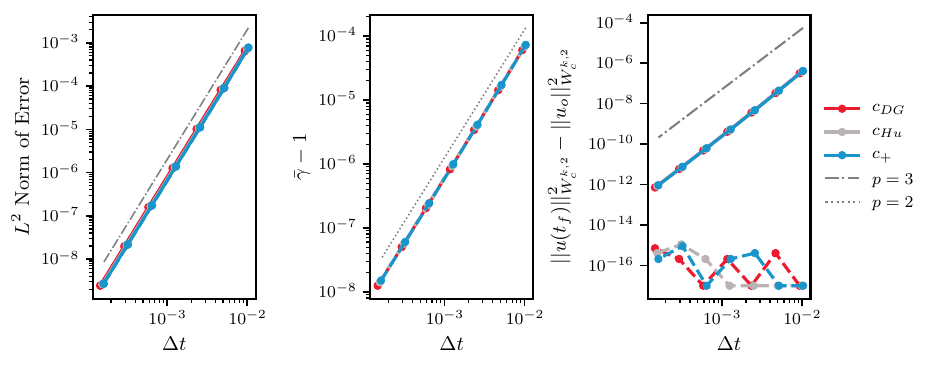}
    \caption{Error convergence for Burgers' equation using various FR correction parameters. The solid lines use SD-NSFR, and the dashed lines use FD-NSFR.  Left: $L^2$ norm of error at the end of the simulation, compared to a calculation using a very small time step. Center: convergence of the average relaxation parameter to $1$ in the FD-NSFR cases. Right: Energy change in the broken Sobolev $W^{k,2}_c$ norm at the end of the calculation. Machine zero is shown as $10^{-17}$ such that it can be plotted on a logarithmic scale. }
    \label{fig:Burgers_FR}
\end{figure}

We present a solution using SSPRK3 and $\Delta t = 0.005$ in Figure \ref{fig:Burgers-evolution} using the $c_{DG}$ correction parameter.
We use the same $p=4$, $32$-element grid as in the temporal convergence study.
Energy is evidently conserved to machine precision by the FD version, with any changes being in increments of machine precision. On the other hand, the SD version is not able to conserve energy. The relaxation parameter is close to $1$ for the entire solution time, though it drops as the shock begins to form at $t=0.3$. 
After the shock forms, the entropy-conserving flux no longer accurately captures the solution, rather it results in a dispersed shock. Therefore the relaxation parameter will deviate as expected from $1$.

\begin{figure}[h!]
\centering
\includegraphics[height=2.5 in]{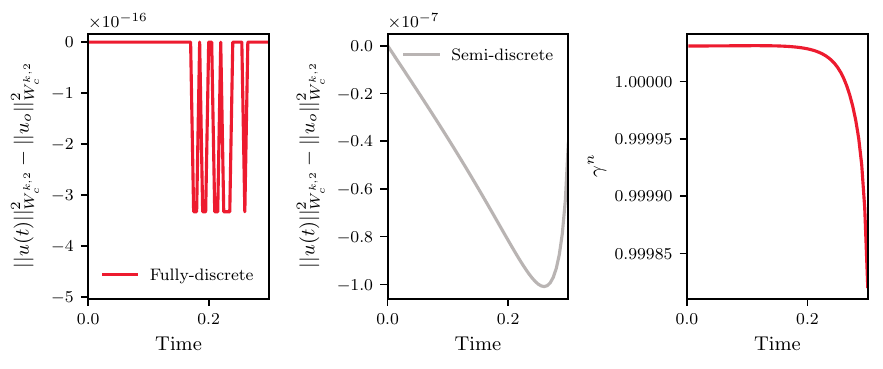}
    \caption{Evolution of energy change using the FD-NSFR (left) and SD-NSFR (center) schemes and $\Delta t = 0.005$. The norms are evaluated in the broken Sobolev norm. Right: Evolution of relaxation parameter $\gamma^n$ using the fully-discrete scheme and $\Delta t = 0.005$.}
    \label{fig:Burgers-evolution}
\end{figure}

Figure \ref{fig:Burgers-M-vs-MpK} demonstrates that NSFR with each correction parameter is indeed entropy conserving in the ${W^{k,2}_c}$ norm, while non-zero parameters $c_{Hu}$ and $c_+$ result in a dissipative energy evolution in the $L^2$ norm as the solution approaches the formation of a discontinuity. 

\begin{figure}[h!]
    \centering\includegraphics[height = 2.5 in]{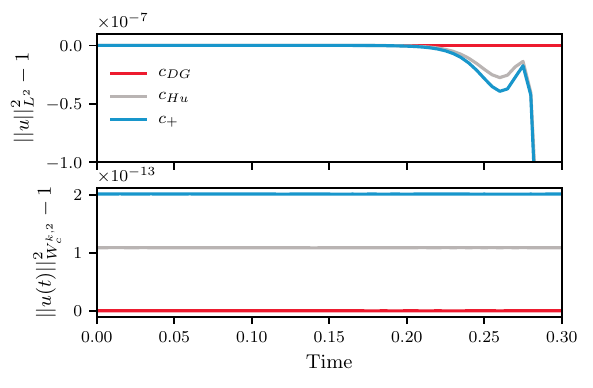}
    \caption{Evolution of energy in the $L^2$ (top) and $W^{k,2}_c$ (bottom) norms for the Burgers' test case using FD-NSFR. Energy is conserved exactly in the $W^{k,2}_c$ norm, but is only conserved in the $L^2$ norm for $c_{DG}$, where $\mathcal{K}=\mathbf{0}$.}
    \label{fig:Burgers-M-vs-MpK}
\end{figure}

\subsection{Results using the viscous Burgers' equation}
We present results using the viscous Burgers' equation to demonstrate that the FD-NSFR scheme effectively follows a dissipative solution. 
We adopted the test case of~\cite[Section 4.4.2]{ketcheson2019relaxation}, but used physical viscous dissipation rather than numerical upwinding. 
We define the PDE and initial condition as
\begin{equation}
    \begin{aligned}
        &f_c(u) = \frac{u^2}{2}, \ \ f_v(u, \nabla u) =  1E-4 \nabla u\\
        &u_0(x) = \exp(-30 (x) ^2) \\
        &x \text{ periodic on } [0, 2],\ \ \ t \in [0,0.2]
    \end{aligned}
    \label{eq:burg-visc}
\end{equation}
The initial condition is smooth and a viscous shock develops during the solution.
We use a grid with 64 evenly-spaced elements with $p=1$ (to mimic the case in~\cite{ketcheson2019relaxation}) with uncollocated GLL flux nodes and GL solution nodes. 
The same two-point flux as Eq.~(\ref{eq:burgers-flux}) is used. We discretize the viscous and solution numerical fluxes with the central viscous approach per Eq.~(\ref{eq:central-viscous}), adding no additional dissipation other than physical viscosity. 
The RK method is set to SSPRK3 and we use the $c_{DG}$ correction parameter. We generate a reference solution using SD-NSFR and a very small time step size. We compare FD-NSFR and SD-NSFR at a larger time step size to the reference solution.
Figure \ref{fig:VBurgers} demonstrates that FD-NSFR follows the reference solution of a viscous shock very well when the viscosity coefficient is small. On the other hand, SD-NSFR results in temporal energy change. 
We also compare a result using the $c_{DG}$ ESFR scheme, which is comparatively more dissipative due to the added Lax-Friedrichs upwinding for stability.

\begin{figure}
    \centering
    \includegraphics[width=0.48\linewidth]{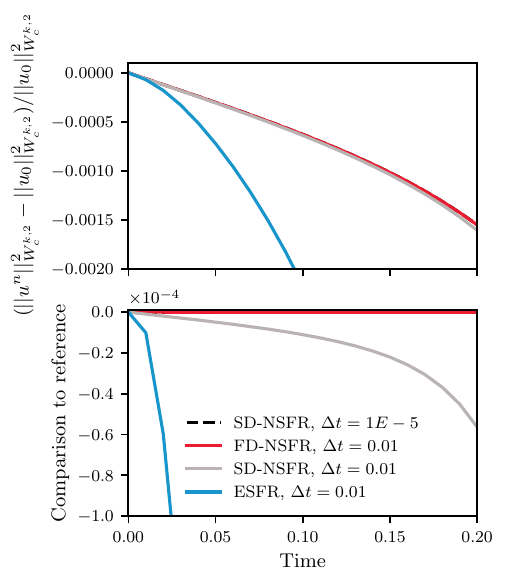}
    \includegraphics[width=0.48\linewidth]{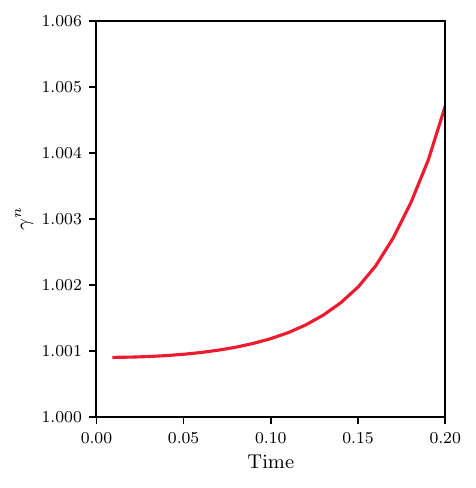}
    \caption{Energy dissipation in the viscous Burgers' test case. The bottom figure subtracts the normalized energy of the reference solution from the indicated large-time-step solution. The left figure shows the relaxation parameter evolution for the FD-NSFR case.}
    \label{fig:VBurgers}
\end{figure}

\newpage
\section{FD-NSFR Scheme for Vector-Valued Conservation Laws} \label{sec:vector-NSFR}

\subsection{Governing equations} \label{sec:NS-euler-governing-eqns}

We consider the compressible Euler and Navier-Stokes equations for a three-dimensional domain. We express the governing law in conservation form [Eq.(~\ref{eq:3d_cons_law})], where the state, and convection and viscous fluxes can be written as 
\begin{equation}
\mathbf{u} = \begin{bmatrix} \rho \\ \rho q_1 \\ \rho q_2 \\ \rho q_3 \\ E \end{bmatrix},   \hspace{0.6cm} 
\mathbf{f}_{c_{i}}\left(\mathbf{u}\right) = \begin{bmatrix}\rho {q}_{i}\\ \rho {q}_{1}{q}_{i} + p\delta_{1i}\\ \rho {q}_{2}{q}_{i} + p\delta_{2i}\\ \rho {q}_{3}{q}_{i} + p\delta_{3i}\\ \left(E + p\right){q}_{i}\end{bmatrix}, \hspace{0.6cm} 
\mathbf{f}_{v_{i}}\left(\mathbf{u},\nabla \mathbf{u}\right) = \begin{bmatrix}0 \\ \tau_{1i} \\ \tau_{2i} \\ \tau_{3i} \\ \tau_{ij}{q}_{j}-\phi_{qi}\end{bmatrix},
\end{equation}
where $i=1,2,3$ denotes the Cartesian components of $\mathbf{x}$ respectively, while $q_i$ denotes the $i$-th component of the velocity vector, $\mathbf{q} = [q_1, q_2, q_3]$. $E$ is the total energy, $E = p/(\gamma_{\text{gas}}-1) + (\rho/2)(q_1^2+q_2^2+q_3^2)$. When we solve the Euler equations, the viscous flux is set $\mathbf{f}_v = \mathbf{0}$. For the Navier-Stokes equations, we define the viscous stress tensor $\bm{\tau}$ as 
\begin{equation}
    \bm{\tau} = \mu\left[2\mathbf{S}+\lambda\left(\nabla\boldsymbol{\cdot}\mathbf{q}\right)\mathbf{I}\right],\hspace{0.6cm}\mathbf{S}=\frac{1}{2}\left[\nabla\mathbf{q}+\left(\nabla\mathbf{v}\right)^{T}\right],
\end{equation}
where $\mathbf{S}$ is the strain-rate tensor, $\mathbf{I}$ is the identity matrix, the bulk viscosity $\lambda=-\frac{2}{3}\mu$ by Stokes' hypothesis, and $\mu$ is the dynamic viscosity. In this work, $\mu=\mu(T)$ is determined by using Sutherland's law \cite{sutherland1893lii}. The heat-flux vector is $\bm{\phi_q}=-\kappa\nabla T$, where $T$ is the temperature, and $\kappa$ is the thermal conductivity:
\begin{equation}
    \kappa = \frac{c_{p}\mu}{\text{Pr}},
\end{equation}
with $\text{Pr}$ as the Prandtl number, set to the standard value of $0.71$ for air. To close the system of equations, the ideal gas law is used,
\begin{equation}
    p = \left(\gamma_{\text{gas}}-1\right)\left(E-\frac{1}{2}\rho \ \mathbf{q}\boldsymbol{\cdot}\mathbf{q}\right),
\end{equation}
\noindent with $\gamma_{\text{gas}}=1.4$ for air. The implemented version of the equation is in non-dimensional form.

The Navier-Stokes equations admit a convex entropy function $S(\mathbf{u}) = -\rho s$, where $s = \log(p \rho^{-\gamma_{\text{gas}}})$ defines the physical entropy for both the Euler and Navier-Stokes equations, following the analysis of Hughes~\cite{hughes1986new}. The entropy variables are defined as the derivative of the convex entropy function $S$ with respect to the conservative variables $\mathbf{u}$, expressed as 
\begin{equation}
    \mathbf{v(\mathbf{u})} = \frac{\gamma_{\text{gas}}-1}{p}\begin{bmatrix} 
    \frac{p}{\gamma_{\text{gas}}-1} ( \gamma_{\text{gas}} + 1 - s ) - E\\ 
    \rho q_1 \\ 
    \rho q_2 \\ 
    \rho q_3 \\ 
    -\rho 
    \end{bmatrix}, 
\end{equation}
where the convexity of $S(\mathbf{u})$ ensures that the mapping  $\mathbf{u(\mathbf{v})}$ is invertible.

\subsection{Relaxation Runge-Kutta for General Convex Quantities} \label{sec:RRK-root}

The root equation (\ref{eq:RRK_explicit}) as presented in section \ref{sec:RRK-overview} can be solved analytically when energy is the numerical entropy variable~\cite{ketcheson2019relaxation}. 
However, the Euler and Navier-Stokes equations do not have a numerical entropy function that can be written in an inner-product form. 
Rather, it is a convex nonlinear function.
In such cases, we adopt the approach of~\cite{ranocha2020relaxation}.
If the numerical entropy is a general convex function, we implement Eq.~(\ref{eq:rrk-root}) as a root-finding problem
\begin{equation}
    r(\gamma^n)= \eta(\mathbf{u}^{n+1}_\gamma) - \eta(\mathbf{u}^{n}) - \gamma^n \Delta t \sum_{i=1}^s b_i \langle \textbf{v}^{(i)}, \frac{d\textbf{u}^{(i)}}{dt} \rangle_{L^2},
    \label{eq:root-implementation}
\end{equation}
\noindent where the total entropy over the domain is evaluated by integrating at solution nodes and summing over the $N$ elements in the domain, 
\begin{equation}
    \eta(\mathbf{u}) = \sum_{m=1}^N \mathbf{1} \mathcal{W} \mathcal{J}_m \mathbf{S}_m^T,
    \label{eq:eta-L2}
\end{equation}
where $\mathbf{S}_m$ is a vector holding the numerical entropy function computed at each quadrature node in the $m$-th element.
To calculate the entropy change estimate, we use an inner product between the entropy and conservative variables in the $L^2$ norm, i.e.,
\begin{equation}
    \langle {\mathbf{v}}^{(i)}, \frac{d\mathbf{u}^{(i)}}{dt} \rangle_{L^2} = \hat{\mathbf{v}}^{(i)\ T} \mathcal{M} \frac{d\hat{\mathbf{u}}^{(i)}}{dt}
\end{equation}
\noindent which uses the unmodified mass matrix. 
The $L^2$ norm is used to ensure consistency with the integrated numerical entropy terms.

We choose to use the secant method in order to numerically find the root of Eq.~(\ref{eq:root-implementation}), aligning with the recommendations of Al Jahdali \textit{et al.}~\cite{aljahdali2022on} and Rogowski \textit{et al.} ~\cite{rogowski2022performance}. 
{\color{black}In some cases, the secant method was unsuccessful in finding a root within the iteration limit, largely due to subtractive cancellation errors.}
We have implemented Algorithm~\ref{alg:root} to find $\gamma^n$, which adds a bisection method solver as a safeguard to prevent the root-solver from failing if the primary secant method solver fails.

\begin{algorithm}
    \caption{Root-finding algorithm.}
    \label{alg:root}
    \begin{algorithmic}
    \Function{root eq.}{$\gamma$}
        \State \Return $\eta(\mathbf{u}^n + \gamma \mathbf{d}) - \eta(\mathbf{u}^n) - \mathbf{e}_{L^2}$  using pre-computed search direction $\mathbf{d}$ and entropy change estimate $\mathbf{e}_{L^2}$ \Comment{Eq.~(\ref{eq:root-implementation})}
    \EndFunction
    \State $tolerance \gets 5E-10; \ \ iterLimit \gets 100$
    \State $\gamma^k \gets \gamma^n + 1E-5$ and $\gamma^{k-1} \gets \gamma^n - 1E-5$
    \State $r_k \gets $ \Call{root eq.}{$\gamma_k$}$;\ \  r_{k-1} \gets $ \Call{root eq.}{$\gamma_{k-1}$}
    \State residual $\gets 1$; $iterCounter \gets 0$
    \While{$residual > tolerance$ and $iterCounter < iterLimit$} \Comment{Secant method loop}
        \State $\gamma_{k+1} \gets \gamma_k - r_k \frac{\gamma_k - \gamma_{k-1}}{r_k-r_{k-1}}$ \Comment{Secant method}
        \If{$\gamma^{k+1} < 0.5$ or $\gamma^{k+1} > 1.5$ or $\gamma^{k+1}$ is NaN}
            \State $\gamma^k = 1 + 10^{-5}$ and $\gamma^{k-1} = 1 - 10^{-5}$ \Comment{reinit from $1$ if current value is far from $1$}
        \EndIf
        \State $\gamma_{k-1} \gets \gamma_k$ and $\gamma_k \gets \gamma_{k+1}$
        \State $r_{k-1} \gets r_k$ and  $r_k \gets$ \Call{root eq.}{$\gamma_k$}
        \State $residual \gets |\gamma_k - \gamma_{k-1}|$
        \State $iterCounter++$
    \EndWhile   
    \If{$iterCounter == iterLimit$ and $residual > tolerance$} \Comment{Bisection method loop}
    \State $\gamma^k \gets \gamma^n + 0.1$ and $\gamma^{k-1} \gets \gamma^n - 0.1$, and reset $iterationCounter = 0$.
    \While{$residual > tolerance$ and $iterationCounter < iterationLimit$}
        \IIf{there is no root in the interval} increase interval size by 0.1\EndIIf
        \State update $\gamma^{k+1}$ with bisection method
        \State $\gamma^k \gets \gamma^{k+1}$ and $\gamma^{k-1} \gets \gamma^k$ 
        \State $residual \gets |\gamma^{k+1} - \gamma^k|$
        \State $iterationCounter++$
    \EndWhile
    \EndIf
    \IIf{$residual > tolerance$} abort \EndIIf
    \State $\gamma^{n+1} \gets \gamma^k$
\end{algorithmic}
\end{algorithm}

The implementation as described by Eq. (\ref{eq:root-implementation}) is similar to that proposed by Ranocha \textit{et al.}~\cite{ranocha2020relaxation}, and the proofs therein apply immediately when the $c_{DG}$ correction parameter is applied and the discretization is stable in the $L^2$ norm. 

\begin{remark}
    However, entropy stability proofs for the NSFR scheme are in the $W_c^{k,2}$ norm~\cite{cicchino2021new,cicchino2022nonlinearly,cicchino2022provably,cicchino2023discretely}.
    In section~\ref{sec:scalar}, we calculated the integrated numerical entropy and entropy change estimate consistently in the $W_c^{k,2}$ norm through the definition of the inner product.
    However, this strategy does not apply to general convex numerical entropy functions since it is not possible to evaluate the integrated numerical entropy in the $W^{k,2}_c$ norm.
    Therefore, we solve for $\gamma^n$ in the $L^2$ norm according to Eq. (\ref{eq:root-implementation}) for any $c$ and the impact of this choice is elaborated as follows.
\end{remark}

\begin{lemma} \label{lem:integrate-num-entr-L2}
{\color{black} The convex entropy function $S(\mathbf{u})$ for the Euler or Navier-Stokes equations is not in inner-product form.}
\end{lemma}

\begin{proof}    
The convex numerical entropy function~\cite[Theorem 1.1]{harten1983symmetric} for the Euler and Navier-Stokes equations is defined as,
\begin{equation}
    S(\mathbf{u}) = \langle {\mathbf{v}}(\mathbf{u}),\mathbf{u} \rangle - q(\mathbf{v}),
    \label{eq:general-numerical-entropy}
\end{equation}
where $S(\mathbf{u})=-\rho s$. Using the definitions for conserved and entropy variables in section~\ref{sec:NS-euler-governing-eqns}, we find that $q(\mathbf{v})=-\rho(\gamma_{\text{gas}}-1)$. 
While the first term of the RHS is in inner-product form and can be evaluated in the $W_c^{k,2}$ norm, the term $q(\mathbf{v})$ cannot be represented as an inner-product. 
\color{black} Hence, it is not possible to integrate Eq.~(\ref{eq:general-numerical-entropy}) in the broken Sobolev $W_c^{k,2}$ norm for the Euler and Navier-Stokes equations. 
\end{proof}

Before proceeding to nonlinear cases, we return to the Burgers' test case in Section~\ref{sec:BurgersInviscid}. 
We stress that it is preferred to use RRK with all terms in the broken Sobolev $W_c^{k,2}$ norm, as was done in Section~\ref{sec:scalar}. 
Using a problem with inner-product numerical entropy allows us to better understand the consequences of controlling the temporal entropy growth in the $L^2$ norm {\color{black} when implementing per Eq.~(\ref{eq:root-implementation})}, while stability for the semi-discrete scheme was established in $W_c^{k,2}$. 

\begin{lemma}
\label{lem:inner-prod-L2}
{\color{black}A semi-discrete entropy conserving scheme in the broken Sobolev $W_c^{k,2}$ norm with a convex entropy function of the form $\eta = 1/2 \langle \mathbf{u}, \mathbf{u} \rangle_{W_c^{k,2}}$ (i.e., energy) leads to energy generation on the order of the RK scheme if temporal energy growth is conserved in $L^2$.
}
\end{lemma}
\begin{proof}
If we apply the root-finding approach for the Burgers' equation with a convex entropy function of the form $\eta = 1/2 \langle \mathbf{u}, \mathbf{u} \rangle_{W_c^{k,2}}$, then equation~\ref{eq:rrk-root} can be expressed as 
\begin{eqnarray}
    r(\gamma) &=& 
        \vert\vert\mathbf{u}^{n+1}_\gamma \vert\vert^2_{W^{k,2}_c} - \vert\vert\mathbf{u}^n \vert\vert^2_{W^{k,2}_c} 
        - 2\gamma^n \Delta t  \sum_{i=1}^s b_i \langle \mathbf{u}^{(i)}, \mathbf{f}(\mathbf{u}^{(i)}) \rangle_{W^{k,2}_c}\nonumber \\
        &=& -2\gamma^n \Delta t^2  \sum_{i,j=1}^s b_i a_{ij}\langle \mathbf{f}(\mathbf{u}^{(j)}), \mathbf{f}(\mathbf{u}^{(i)}) \rangle_{L^2}
    + (\gamma^n)^2 \Delta t^2  \sum_{i,j=1}^s b_i b_{j}\langle \mathbf{f}(\mathbf{u}^{(j)}), \mathbf{f}(\mathbf{u}^{(i)}) \rangle_{L^2} \nonumber \\
    && -2\gamma^n \Delta t^2  \sum_{i,j=1}^s b_i a_{ij}\langle \mathbf{f}(\mathbf{u}^{(j)}), \mathbf{f}(\mathbf{u}^{(i)}) \rangle_{\mathcal{K}}
    + (\gamma^n)^2 \Delta t^2  \sum_{i,j=1}^s b_i b_{j}\langle \mathbf{f}(\mathbf{u}^{(j)}), \mathbf{f}(\mathbf{u}^{(i)}) \rangle_{\mathcal{K}},
    \label{eq:rrk-root-expanded-norms}
\end{eqnarray}
where the inner-product in the $W^{k,2}_c$ is expanded as $\langle \mathbf{a} , \mathbf{b} \rangle_{W^{k,2}_c} =\langle \mathbf{a} , \mathbf{b} \rangle_{L^2}+\langle \mathbf{a} , \mathbf{b} \rangle_{\mathcal{K}}$. The four terms on the right-hand-side are of the order of the RK scheme and contribute towards the temporal entropy growth. Note that if the root-finding algorithm was solved with all inner-products evaluated in $L^2$, as in Eq.~(\ref{eq:root-implementation}), then a relaxation parameter $\gamma$ would be found only to eliminate the first two terms on the right-hand-side, leaving the final two terms to grow at the order of the temporal scheme.
\end{proof}

{\color{black}Lemma~\ref{lem:inner-prod-L2} establishes that the NSFR approach is not fully-discretely entropy conversing in the $W^{k,2}_c$ norm and sees temporal entropy growth at the order the RK schemes for terms computed in the seminorm, $\mathcal{K}$. In the following algorithm, we offer a solution by which the scheme is still fully-discrete in the temporal $L^2$ sense and we will further show in the results section where the scheme still benefits from robustness offered by fully-discrete entropy-conserving schemes and continues to provide stable solutions at higher CFL numbers compared to SD-NSFR. Since the scheme is semidiscretely entropy conserving in $W^{k,2}_c$ while fully-discrete in the temporal discretization in $L^2$, we report total entropy growth through the following definition.}

\begin{definition} \label{def:NSFR_num_entropy}
    For general convex functions that are not in an inner-product form, the ``NSFR numerical entropy" $\eta_c$ is defined as
    \begin{equation}
        \eta_c({\mathbf{u}}^{n+1}_\gamma) = \eta({\mathbf{u}}^{n+1}_\gamma) + \gamma^n \Delta t  \sum_{i=1}^s b_i \langle {\mathbf{v}}({\mathbf{u}}^{(i)}), \frac{\partial {\mathbf{u}}^{(i)}}{\partial t} \rangle_{\mathcal{K}},
    \end{equation}
    where the solution {\color{black} $\mathbf{u}$} is obtained from a semidiscretization which has a stability guarantee in the broken Sobolev $W_c^{k,2}$ norm, alike Equation~(\ref{eq:NSFR-vector}).
    Using $\eta_c$, we observe fully-discrete entropy conservation for the Burgers' and Euler equations in the sense that the spatial discretization is conserving in the broken Sobolev $W_c^{k,2}$ norm and temporal entropy change is prevented in $L^2$.
\end{definition}

\begin{algorithm}
    \caption{Relaxation RK Approach for NSFR with General Convex Entropy Functions}
    \label{alg:rrkForNavierStokes}
    \begin{algorithmic}
    \Function{RHS}{$\mathbf{u}$} 
        \State \Return result of Eq.~(\ref{eq:NSFR-vector}) \Comment{NSFR semidiscretization}
    \EndFunction
    \For{i in 1:s}
        \State $\mathbf{u}^{(i)} \gets \mathbf{u}^{n} + \Delta t^n \sum_{j=0}^s a_{ij} $ \Call{RHS}{$\mathbf{u}^{(i)}$}
    \EndFor

    \State $\mathbf{d} \gets \Delta t^n \sum_{i=0}^{s} b_i $\Call{RHS}{$\mathbf{u}^{(i)}$} \Comment{Search direction}
    \State $\mathbf{e}_{L^2} \gets \Delta t^n \sum_{i=1}^s b_i \langle \textbf{v}^{(i)}, $ \Call{RHS}{$\mathbf{u}^{(i)}$} $\rangle_{L^2} $ \Comment{$L^2$ entropy change estimate}
    \State Solve for $\gamma^n$ from Algorithm~\ref{alg:root}, using $\mathbf{e}_{L^2}$ and $\mathbf{d}$ to define \Call{root eq.}{$\gamma$} 
    \State $\mathbf{u}^{n+1}_\gamma \gets \mathbf{u}^n + \gamma^n \mathbf{b}$ 
    \State $t^{n+1}_\gamma \gets t^n + \gamma^n \Delta t$
    \State $\eta_c \gets \eta(\mathbf{u}^{n+1}_\gamma)+ \gamma^n \Delta t^n \sum_{i=1}^s b_i \langle \textbf{v}^{(i)}, $ \Call{RHS}{$\mathbf{u}^{(i)}$} $\rangle_{\mathcal{K}}$ \Comment{Per definition~\ref{def:NSFR_num_entropy}}
    \IIf{CFL adaptation is used} update $\Delta t^{n+1}$\EndIIf
\end{algorithmic}
\end{algorithm}

 Figure~\ref{fig:RRK-Burg-root-NRG-ev} confirms that if we employ Algorithm~\ref{alg:rrkForNavierStokes} then growth in energy evaluated based on $\eta_c$ is maintained at machine precision, while the impact of the $\mathcal{K}$ terms are evident when energy is evaluated in the $W^{k,2}_c$ norm.

\begin{figure}
    \centering
    \includegraphics[height=2.5 in]{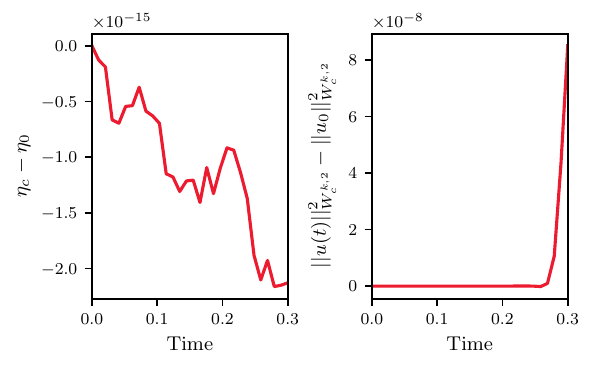}
    \caption{Energy evolution in the inviscid Burgers' case using the root equation Eq.~(\ref{eq:root-implementation}) and $c_+$, using $29$ time steps. The cumulative $L^2$ energy change is conserved (right), but broken Sobolev $W_c^{k,2}$ energy is not conserved (left).}
    \label{fig:RRK-Burg-root-NRG-ev}
\end{figure}

We demonstrate in Figure~\ref{fig:RRK-Burg-rootvsalg} that the energy generation in the $\mathcal{K}$  norm is {\color{black}smaller by approximately an order of magnitude} compared to that in the $W^{k,2}_c$ norm, if no correction had been made at all. 
We use the root-finding variation of RRK, implemented using the secant method solver defined by Algorithm~\ref{alg:root} and compare convergence using SSPRK3 in Figure \ref{fig:RRK-Burg-rootvsalg}. The problem definition is the same as in Section~\ref{sec:scalar}. We confirm that the error and relaxation parameter converge nearly identically between all variants. 
The root-finding approach for $c_{DG}$ shifts the energy generation up due to the imperfect convergence of the root solver. 
Such a result demonstrates that Eq. (\ref{eq:root-implementation}) indeed fully-discretely conserves numerical entropy to the level of the convergence tolerance only if the correction parameter is chosen to be $c_{DG}$.

\begin{figure}
    \centering
    \includegraphics[height=2.5 in]{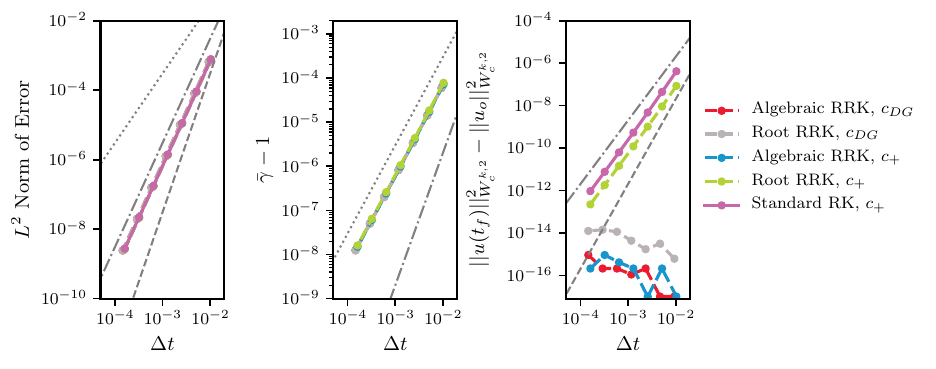}
    \caption{Convergence of the inviscid Burgers' test case using the root-finding RRK method and the algebraic RRK method. All results use SSPRK3. ``Algebraic RRK" indicates that the method introduced in Section~\ref{sec:scalar} is used to find $\gamma^n$, ``Root RRK" indicates that we solve using Eq.~(\ref{eq:root-implementation}), and ``Standard RK" indicates that no relaxation approach is used.}
    \label{fig:RRK-Burg-rootvsalg}
\end{figure}


We now return to the case where the numerical entropy function is a general convex function, such as that for the Euler and Navier-Stokes equations, where $S(\mathbf{u}) = -\rho s$.
\color{black}
{ For cases where the spatial entropy guarantee is in $L^2$, Ranocha \textit{et al.}~\cite{ranocha2020relaxation} established the existence and accuracy of a relaxation parameter $\gamma^n$ by showing that the root function is convex with a root at $0$ and another close to $1$. The analysis of~\cite{ranocha2020relaxation} applies directly if the $c_{DG}$ correction parameter is used. We discretely demonstrate that the root function has similar properties regardless of the choice of correction parameter} for the inviscid Taylor-Green vortex test case, which will be introduced in section \ref{sec:inv-tgv}, as an entropy-conserving Euler test case.
We demonstrate discretely that the root function is convex with a root close to $1$ in Figure~\ref{fig:RRK-root-Euler} (a). 
This is true for both $c_{DG}$ and $c_+$ using Algorithm~\ref{alg:rrkForNavierStokes}. Moreover, the choice of $c$ does not change the form of the root equation, reflected in the comparable shapes of the root equation for both correction parameters.
To confirm the convexity of the root function, we calculate the slope of the root equation at every time step using finite difference. 
The slope is strictly negative at $\gamma=0$ and strictly positive at $\gamma=1$ for all time using either correction parameter as observed in Fig.~\ref{fig:RRK-root-Euler} (b).
This is true regardless of the sign and magnitude of the $L^2$-norm entropy change estimate, seen in the bottom plot of Figure \ref{fig:RRK-root-Euler} (b). 
Furthermore, we will provide confirmation of the convergence orders in subsequent sections for both $c$ parameters in a variety of test cases.

\begin{figure}[h!]
\centering
\begin{subfigure}[t]{0.48\textwidth}
    \centering
        \includegraphics[height = 5 in]{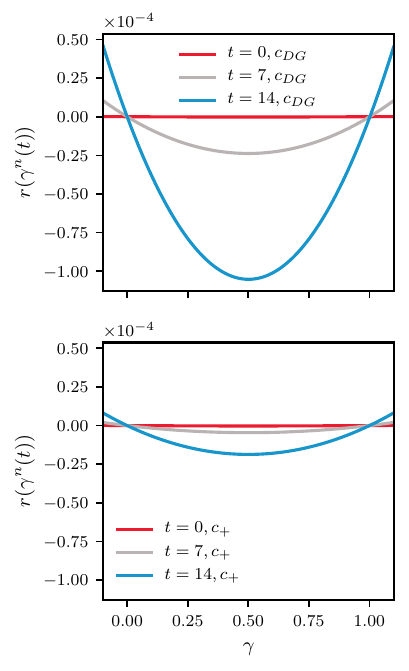}
    \caption{The root function plotted as a function of $\gamma$ for the inviscid TGV test case for $c_{DG}$ (top) and $c_+$ (bottom). }
\end{subfigure}
\hfill
\begin{subfigure}[t]{0.48\textwidth}
    \centering
    \includegraphics[height = 5 in]{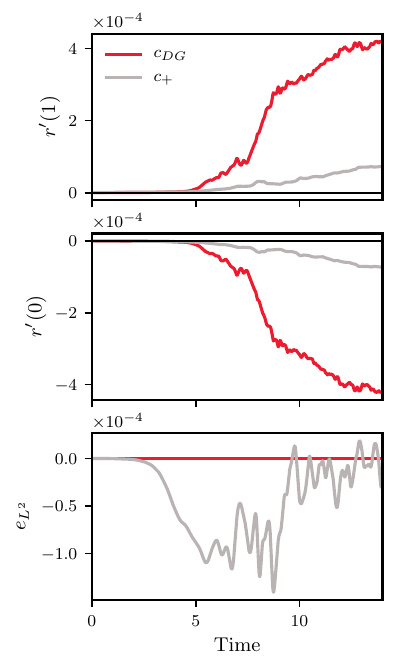}
    \caption{The slope of the root function  at $\gamma=0$ (top), at $\gamma=1$ (middle), and the entropy change estimate $e_{L^2}\coloneqq \Delta t \sum_{i=1}^s b_i \langle \textbf{v}^{(i)}, \frac{d\textbf{u}^{(i)}}{dt} \rangle_{L^2}$(bottom), plotted over the simulation time. Even in the low-$t$ region, $r^\prime(1)>0$ and $r^\prime(0)<0$.}
\end{subfigure}
    \caption{The root function for $c_{DG}$ and for $c_+$, plotted per the implementation in Eq.~(\ref{eq:root-implementation}). The test case is as described in Section \ref{sec:inv-tgv}, using the largest stable CFL, i.e. $CFL=0.48$ for the $c_{DG}$ version, and $CFL = 0.54$ for the $c_+$ version.}
    \label{fig:RRK-root-Euler}
\end{figure}

\subsection{Results using the isentropic vortex advection test case}

The following set of numerical experiments will use the compressible Euler equations as introduced in Section \ref{sec:NS-euler-governing-eqns}. For all Euler and Navier-Stokes test cases, we choose the entropy-conserving flux of Chandrashekar~\cite{chandrashekar2013kinetic}, with the modification of Ranocha~\cite{ranocha2021preventing} to ensure pressure equilibrium. 
In entropy-stable cases, we add local Lax-Friedrichs dissipation as introduced in Section~\ref{sec:num-flux-choice} to the entropy-conserving flux, where the maximum eigenvalue at each surface is calculated as
\begin{equation}
    \lambda_{\text{max}}= c + |{q}_i \hat{n}_i|,
\end{equation}
\noindent where $c=\sqrt{p \gamma_{\text{gas}} / \rho}$ is the speed of sound and $\hat{n}_i$ are the components of the unit normal vector.

The two-dimensional isentropic vortex test case is used to ensure that the scheme achieves an acceptable order of convergence when solving the Euler equations. 
We use the Shu variant of the isentropic vortex advection case~\cite{shu1998essentially}, implemented nondimensionally according to Table 1 in~\cite{spiegel2015survey}. The domain size is increased from $L=5$ to $L=10$ according to the recommendations of Spiegel \textit{et al.}~\cite{spiegel2015survey}. 
The exact solution for $[x,y] \in [-L,L], t  \in [0,t_f]$ in primitive variables is 
\begin{equation}
    \begin{aligned}
        \bar{x}(x, t) &= \mod \left(x - M_\infty \cos(\alpha) t + L, 2L\right) - L \\
        \bar{y}(y, t) &= \mod \left(y - M_\infty \sin(\alpha) t + L, 2L\right) - L \\
        \varphi(x, y, t) &= M_\infty \frac{5 \sqrt{2}}{4 \pi} 
                     \exp{ \left( \frac{-1}{4} 
                     \left( \bar{x}(x, t)^2 + \bar{y}(y, t)^2
                     \right)
                     \right)
                     }\\
        \rho(x,y,t) &= \left( 1 - \frac{\gamma_{\text{gas}}-1}{2} \varphi(x,y,t)^2
                     \right)
                     ^{1/(\gamma_{\text{gas}}-1)}\\
        q_1(x,y,t)  &= M_\infty \cos(\alpha) - \bar{y}(y, t) \varphi(x,y,t) \\
        q_2(x,y,t)  &= M_\infty \sin(\alpha) + \bar{x}(x, t) \varphi(x,y,t) \\
        p(x,y,t) &= \frac{1}{\gamma_{\text{gas}}} 
                     \left( 1 - \frac{\gamma_{\text{gas}}-1}{2} \varphi(x,y,t)^2
                     \right)
                     ^{\gamma_{\text{gas}}/(\gamma_{\text{gas}}-1)},
    \end{aligned}
\end{equation}
\noindent where $\bar{x}, \bar{y}$ translate the vortex in time and take into account the periodic boundaries, and $\varphi$ is the Gaussian distribution of the isentropic vortex.
The Mach number is $M_\infty = \sqrt{2/\gamma_{\text{gas}}}$, angle of attack is $\alpha = \pi/4$, ratio of specific heats is $\gamma_{\text{gas}} = 1.4$, and half-length of the domain is $L = 10$. 
We calculate for a single cycle through the domain until the vortex returns to its initial location. 
The exact solution is always calculated as a function of the exact end time, which is necessary when using FD-NSFR due to the adjustment of the time step size.
We use Cartesian grids with equally-sized elements. The solution and flux nodes are collocated GLL.
The constant time step size is set as $\Delta t = \frac{\Delta x}{10 M_\infty (p+1)}$, where $\Delta x$ is the grid size. 

We choose to evaluate pressure convergence as it accounts for both velocity and density and, therefore may be a better indicator of overall performance than density error. 
The error is calculated by overintegrating by 10 over the domain. 
We provide the convergence behaviour for combinations of spatial orders $p=2$ or $p=3$ and SSPRK3 and RK4 respectively, correction parameters $c_{DG}$, $c_{Hu}$ or $c_+$, and both entropy-conserving and entropy-stable fluxes.
The convergence behaviour is shown in Figure \ref{fig:isentropic-convergence}, and tabulated in Table \ref{tab:isentropic_conv_ec} for entropy-conserving cases and Table \ref{tab:isentropic_conv_es} for entropy-stable cases.

The entropy-conserving schemes converge at about $p$ for all FR schemes.
When Lax-Friedrichs dissipation is added to form an entropy-stable FD-NSFR scheme in Table \ref{tab:isentropic_conv_es}, the convergence order is closer to optimal $p+1$ convergence for the $p=3$ and RK4 case, however the $L^\infty$ error remains at about $p + 1/2$.
As expected, a slightly upward shift is observed in the error levels when the FR correction parameter $c$ is increased; however, we observe the same convergence order.
Sub-optimal convergence in entropy-conserving schemes is consistent with current literature: Crean \textit{et al.}~\cite{crean2018entropy} show convergence lower than $p+1$ for the isentropic vortex case using both entropy-conserving and entropy-stable fluxes, while Chan~\cite{chan2018discretely} shows $p+1$ convergence for small $p$ and $p+1/2$ convergence for a higher $p$ while using an entropy-stable scheme for the isentropic vortex case.
As the convergence results presented herein are consistent with established literature, we deem the accuracy of the FD-NSFR scheme to be satisfactory.

\begin{figure}
\begin{minipage}{0.4 \textwidth}
    \centering
    \includegraphics[height = 2.5 in]{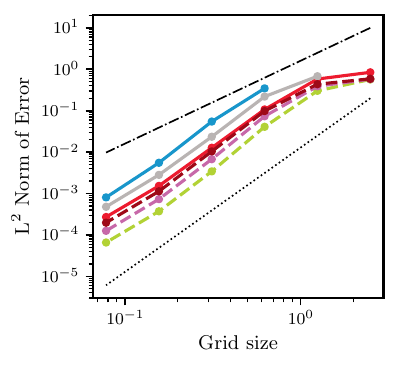}
\end{minipage}
\begin{minipage}{0.6 \textwidth}
    \centering
    \includegraphics[height = 2.5 in]{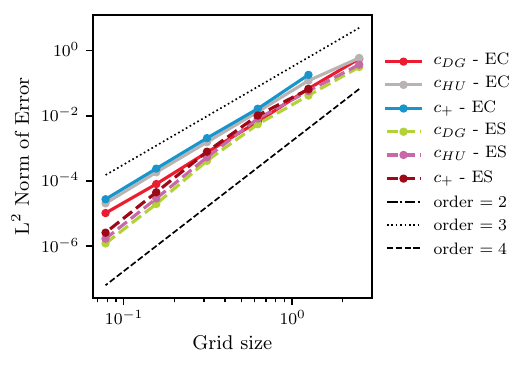}
\end{minipage}
    \caption{$L^2$ error evaluated using pressure for the isentropic vortex test case for the FD-NSFR scheme. Left: convergence when $p=2$ polynomial basis functions are used with SSPRK3 temporal integration. Right: convergence with $p=3$ and RK4. Entropy-conserving fluxes are denoted ``EC", while entropy-stable fluxes are denoted ``ES".}
    \label{fig:isentropic-convergence}
\end{figure}

\begin{table}[hbt!]
\centering
\caption{Convergence of pressure to the exact solution after one cycle of the isentropic vortex advection case using the fully-discrete method with the entropy-conserving Ranocha flux.}
\label{tab:isentropic_conv_ec}
\begin{tabular}{|p{0.75 in}|l|ll|ll|ll|}
\hline
Scheme &  $n$   & $L^1$ error & $L^1$ rate & $L^2$ error & $L^2$ rate & $L^\infty $ error & $L^\infty$ rate \\
\hline
     $c_{DG}$,&   8 &           1.06E+01 &                      - &           8.44E-01 &                      - &               5.33E-01 &                          - \\
SSPRK3, &  16 &           5.61E+00 &                   0.92 &           5.82E-01 &                   0.54 &               3.99E-01 &                       0.42 \\
  $p=2$ &  32 &           1.18E+00 &                   2.25 &           1.07E-01 &                   2.44 &               7.98E-02 &                       2.32 \\
        &  64 &           1.80E-01 &                   2.71 &           1.27E-02 &                   3.07 &               1.08E-02 &                       2.89 \\
        & 128 &           2.18E-02 &                   3.04 &           1.53E-03 &                   3.06 &               1.36E-03 &                       2.98 \\
        & 256 &           3.80E-03 &                   2.52 &           2.71E-04 &                   2.50 &               3.28E-04 &                       2.05 \\
\hline
     $c_{Hu}$,&  16 &           7.64E+00 &                      - &           6.79E-01 &                      - &               4.56E-01 &                          - \\
SSPRK3, &  32 &           2.23E+00 &                   1.78 &           2.19E-01 &                   1.63 &               1.47E-01 &                       1.64 \\
  $p=2$ &  64 &           3.19E-01 &                   2.80 &           2.35E-02 &                   3.22 &               1.47E-02 &                       3.32 \\
        & 128 &           4.01E-02 &                   2.99 &           2.82E-03 &                   3.06 &               2.28E-03 &                       2.69 \\
        & 256 &           6.41E-03 &                   2.65 &           4.78E-04 &                   2.56 &               4.88E-04 &                       2.23 \\
\hline
     $c_{+}$,&  32 &           4.07E+00 &                      - &           3.46E-01 &                      - &               2.19E-01 &                          - \\
SSPRK3, &  64 &           6.32E-01 &                   2.69 &           5.47E-02 &                   2.66 &               3.75E-02 &                       2.54 \\
  $p=2$ & 128 &           7.09E-02 &                   3.16 &           5.53E-03 &                   3.31 &               5.35E-03 &                       2.81 \\
        & 256 &           9.87E-03 &                   2.84 &           8.07E-04 &                   2.78 &               9.33E-04 &                       2.52 \\
\hhline{|=|=|==|==|==|}
    $c_{DG}$,&   8 &           6.81E+00 &                      - &           5.66E-01 &                      - &               3.74E-01 &                          - \\
  RK4, &  16 &           9.31E-01 &                   2.87 &           6.80E-02 &                   3.06 &               3.71E-02 &                       3.33 \\
 $p=3$ &  32 &           9.64E-02 &                   3.27 &           6.61E-03 &                   3.36 &               3.88E-03 &                       3.26 \\
       &  64 &           1.06E-02 &                   3.19 &           7.62E-04 &                   3.12 &               5.11E-04 &                       2.93 \\
       & 128 &           1.07E-03 &                   3.31 &           8.10E-05 &                   3.23 &               5.79E-05 &                       3.14 \\
       & 256 &           1.20E-04 &                   3.15 &           1.04E-05 &                   2.96 &               1.16E-05 &                       2.32 \\
 \hline
    $c_{Hu}$,&   8 &           6.28E+00 &                      - &           5.96E-01 &                      - &               3.94E-01 &                          - \\
  RK4, &  16 &           1.67E+00 &                   1.91 &           1.20E-01 &                   2.31 &               7.21E-02 &                       2.45 \\
 $p=3$ &  32 &           1.93E-01 &                   3.11 &           1.33E-02 &                   3.17 &               8.77E-03 &                       3.04 \\
       &  64 &           2.21E-02 &                   3.12 &           1.56E-03 &                   3.09 &               1.12E-03 &                       2.97 \\
       & 128 &           2.46E-03 &                   3.17 &           1.86E-04 &                   3.07 &               1.56E-04 &                       2.85 \\
       & 256 &           2.61E-04 &                   3.24 &           2.07E-05 &                   3.16 &               2.15E-05 &                       2.85 \\
\hline
    $c_{+}$,&  16 &           2.57E+00 &                      - &           1.81E-01 &                      - &               1.01E-01 &                          - \\
  RK4, &  32 &           2.40E-01 &                   3.42 &           1.65E-02 &                   3.45 &               9.49E-03 &                       3.41 \\
 $p=3$ &  64 &           2.91E-02 &                   3.04 &           2.07E-03 &                   2.99 &               1.87E-03 &                       2.35 \\
       & 128 &           3.09E-03 &                   3.24 &           2.40E-04 &                   3.11 &               2.59E-04 &                       2.85 \\
       & 256 &           3.24E-04 &                   3.26 &           2.76E-05 &                   3.12 &               4.44E-05 &                       2.55 \\
\hline
\end{tabular}
\end{table}

\begin{table}[hbt!]
\centering
\caption{Convergence of pressure to the exact solution fter one cycle of the isentropic vortex advection case using the fully-discrete method with the Ranocha flux with added Lax-Friedrichs dissipation.}
\label{tab:isentropic_conv_es}
\begin{tabular}{|p{0.75 in}|l|ll|ll|ll|}
\hline
Scheme &  $n$   & $L^1$ error & $L^1$ rate & $L^2$ error & $L^2$ rate & $L^\infty $ error & $L^\infty$ rate \\
     $c_{DG}$,&   8 &           3.27E+00 &                      - &           5.67E-01 &                      - &               4.09E-01 &                          - \\
SSPRK3, &  16 &           1.36E+00 &                   1.27 &           3.02E-01 &                   0.91 &               2.40E-01 &                       0.77 \\
  $p=2$ &  32 &           1.80E-01 &                   2.92 &           4.09E-02 &                   2.88 &               3.69E-02 &                       2.70 \\
        &  64 &           1.57E-02 &                   3.52 &           3.43E-03 &                   3.58 &               3.34E-03 &                       3.47 \\
        & 128 &           1.85E-03 &                   3.08 &           3.72E-04 &                   3.20 &               5.94E-04 &                       2.49 \\
        & 256 &           2.91E-04 &                   2.67 &           6.57E-05 &                   2.50 &               9.78E-05 &                       2.60 \\
\hline
     $c_{Hu}$,&   8 &           3.83E+00 &                      - &           5.78E-01 &                      - &               4.06E-01 &                          - \\
SSPRK3, &  16 &           1.69E+00 &                   1.17 &           3.89E-01 &                   0.57 &               3.01E-01 &                       0.43 \\
  $p=2$ &  32 &           3.22E-01 &                   2.40 &           7.35E-02 &                   2.41 &               7.12E-02 &                       2.08 \\
        &  64 &           2.91E-02 &                   3.47 &           6.75E-03 &                   3.44 &               7.65E-03 &                       3.22 \\
        & 128 &           3.57E-03 &                   3.03 &           7.32E-04 &                   3.21 &               1.04E-03 &                       2.88 \\
        & 256 &           5.69E-04 &                   2.65 &           1.26E-04 &                   2.54 &               1.87E-04 &                       2.48 \\
\hline
     $c_{+}$,&   8 &           4.50E+00 &                      - &           5.89E-01 &                      - &               4.00E-01 &                          - \\
SSPRK3, &  16 &           1.96E+00 &                   1.20 &           4.36E-01 &                   0.43 &               3.35E-01 &                       0.26 \\
  $p=2$ &  32 &           4.73E-01 &                   2.05 &           9.80E-02 &                   2.15 &               9.19E-02 &                       1.87 \\
        &  64 &           4.79E-02 &                   3.30 &           1.04E-02 &                   3.24 &               1.16E-02 &                       2.99 \\
        & 128 &           6.09E-03 &                   2.97 &           1.13E-03 &                   3.20 &               1.76E-03 &                       2.72 \\
        & 256 &           1.05E-03 &                   2.53 &           1.98E-04 &                   2.52 &               3.13E-04 &                       2.49 \\
\hhline{|=|=|==|==|==|}
     $c_{DG}$, &   8 &           1.87E+00 &                      - &           3.16E-01 &                      - &               2.38E-01 &                          - \\
  RK4, &  16 &           4.15E-01 &                   2.17 &           4.22E-02 &                   2.90 &               2.96E-02 &                       3.01 \\
 $p=3$ &  32 &           7.76E-02 &                   2.42 &           5.52E-03 &                   2.93 &               3.39E-03 &                       3.12 \\
       &  64 &           5.95E-03 &                   3.70 &           4.19E-04 &                   3.72 &               3.32E-04 &                       3.35 \\
       & 128 &           2.52E-04 &                   4.56 &           1.98E-05 &                   4.40 &               2.15E-05 &                       3.95 \\
       & 256 &           1.07E-05 &                   4.56 &           1.22E-06 &                   4.02 &               2.14E-06 &                       3.33 \\
 \hline
     $c_{Hu}$, &   8 &           2.17E+00 &                      - &           3.72E-01 &                      - &               2.72E-01 &                          - \\
  RK4, &  16 &           6.14E-01 &                   1.82 &           5.80E-02 &                   2.68 &               4.07E-02 &                       2.74 \\
 $p=3$ &  32 &           1.10E-01 &                   2.49 &           8.22E-03 &                   2.82 &               5.81E-03 &                       2.81 \\
       &  64 &           6.86E-03 &                   4.00 &           5.39E-04 &                   3.93 &               3.70E-04 &                       3.97 \\
       & 128 &           3.14E-04 &                   4.45 &           2.91E-05 &                   4.21 &               3.16E-05 &                       3.55 \\
       & 256 &           1.32E-05 &                   4.57 &           1.70E-06 &                   4.10 &               2.80E-06 &                       3.50 \\
\hline
      $c_{+}$ &  16 &           6.23E-01 &                      - &           6.57E-02 &                      - &               4.70E-02 &                          - \\
  RK4, &  32 &           1.33E-01 &                   2.23 &           1.02E-02 &                   2.68 &               7.79E-03 &                       2.59 \\
 $p=3$ &  64 &           9.19E-03 &                   3.85 &           7.95E-04 &                   3.69 &               5.96E-04 &                       3.71 \\
       & 128 &           4.15E-04 &                   4.47 &           4.53E-05 &                   4.13 &               5.46E-05 &                       3.45 \\
       & 256 &           1.67E-05 &                   4.63 &           2.61E-06 &                   4.12 &               4.64E-06 &                       3.56 \\
\hline
\end{tabular}
\end{table}

\subsection{Results using the inviscid Taylor-Green vortex test case}
The inviscid Taylor-Green vortex test case is a challenging test case for standard DG methods, especially when coarse grids are used.
We use the Euler equations as defined in Section \ref{sec:NS-euler-governing-eqns} and initialize per~\cite{wang2013high}, the initial condition used in the 1st International Workshop on High-Order CFD Methods:
\begin{equation}
    \begin{aligned}
        q_1 &= \sin\left(x\right) \cos\left(y \right)\cos\left(z\right) \\ 
        q_2 &= - \cos\left(x\right) \sin\left(y\right) \cos\left(z\right) \\
        q_3 &= 0 \\
        p &=   \frac{1}{\gamma_{\text{gas}} M_o^2} + \frac{1}{16} \left(\cos\left(2x \right) + \cos\left(2y \right)\right) \left(\cos\left(2z \right) +2\right),
    \end{aligned}
    \label{eq:tgv-initial-cond}
\end{equation}
where $q$ are velocities in each coordinate direction. 
We use an isothermal density initialization, setting $\rho = p \gamma_{\text{gas}} M_o^2$ with $M_o = 0.1$.
We solve in the triply-periodic domain $\mathbf{x} \in [0, 2\pi]^3, \ t \in [0,14]$. 
The time step is chosen adaptively according to the current solution state as
\begin{equation}
    \Delta t^n = \frac{CFL \ \Delta x }{\lambda_{\text{max}}^n (p+1)},
    \label{eq:CFL-adaptation}
\end{equation}
where $CFL$ is the Courant-Friedrichs-Lewy number, $\Delta x$ is the element length, and $\lambda_{\text{max}}^n=\sqrt{q_1^2 + q_2^2 + q_3^2} + \sqrt{p \gamma_{\text{gas}}/\rho}$ is the maximum wave speed at each step.
We use SSPRK3 for the RK method, and the root-finding version of RRK as described in Section~\ref{sec:RRK-root} for the FD-NSFR formulation.
When the FD-NSFR scheme is used, the adapted time step size per Eq.~(\ref{eq:CFL-adaptation}) is subsequently modified by the relaxation parameter.
We use a coarse grid with 8 elements per direction with $p=3$.
As with the isentropic vortex test case, we use collocated GLL solution and flux nodes and straight, evenly-spaced elements, and employ Ranocha's modification to the Chandrashekar flux~\cite{ranocha2021preventing, chandrashekar2013kinetic}.

Figure~\ref{fig:TGV-entropy-ev} compares FD-NSFR and SD-NSFR schemes at large time step sizes. 
{\color{black}The largest time step is determined by finding the maximum time step size that does not cause an abort due to negative density. This is achieved through trial and error by incrementally increasing the CFL number by $0.01$ until the simulation fails.}
Two refinements by a factor of two are shown for both $c_{DG}$ and $c_+$.
{\color{black} When the FD-NSFR scheme is used, 
NSFR numerical entropy is near machine zero for any time step size for both $c_+$ and $c_{DG}$.
This verifies the entropy-conserving property for $c_{DG}$ in the $L^2$ norm. For the non-zero $c_+$, we confirm that there is no temporal entropy change in the $L^2$ norm. 
On the other hand, SD-NSFR results in an entropy change on the order of $10^{-7}$ for $c_{DG}$ or $10^{-8}$ for $c_+$.
}
While $c_+$ SD-NSFR has a smaller magnitude of entropy change than $c_{DG}$ SD-NSFR, the change remains orders of magnitude higher than the FD-NSFR version using the same correction parameter.


\begin{figure}[h]
    \centering
    \includegraphics[height=2.5in]{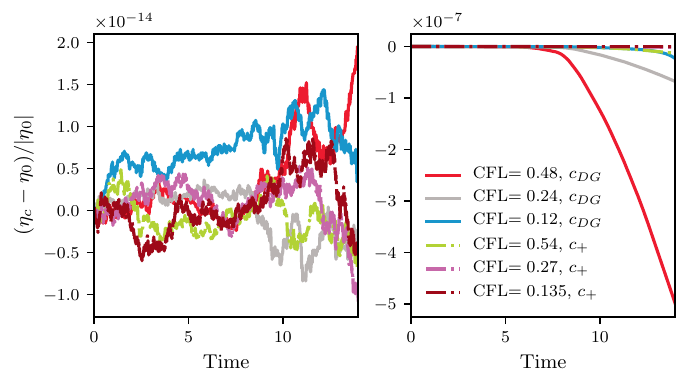}
    \caption{Evolution of cumulative $\eta_c$ for the inviscid TGV test case using FD-NSFR (left) and SD-NSFR (right). Three refinements by a factor of two are shown from the largest stable time step. Entropy change is normalized by the value at $t=0$.}
    \label{fig:TGV-entropy-ev}
\end{figure}

We also plot the relaxation parameter evolution during the solution for the FD-NSFR schemes in Figure~\ref{fig:TGV-relaxation-param}. 
Until $t=4.5$, the solution remains smooth and uniform, attributed to the sinusoidal nature of the initial condition.
As time advances, the flow gradually loses its structured pattern, yielding to the presence of high-frequency modes.  
This is mirrored by the relaxation parameter: before $t=4.5$, the relaxation parameter oscillates close to $1$, and departs after $t=4.5$; however it remains within ${\cal O}(\Delta t^{p-1})$.
We also note that the $c_+$ calculations have a smaller magnitude of relaxation parameter, especially in the second regime. 
This observation can be associated with the lesser extent of entropy change evident in Figure~\ref{fig:TGV-entropy-ev}. 
Since FR correction parameters larger than $c_{DG}$ damp the modal coefficients associated with the highest order of the solution polynomial, we conjecture that the scheme subsequently suppresses entropy growth associated with these modes. As a result, correction parameters larger than $c_{DG}$ result in a smaller correction required to prevent the growth of entropy from the temporal discretization.

\begin{figure}[h]
    \centering
    \includegraphics[height=2.5in]{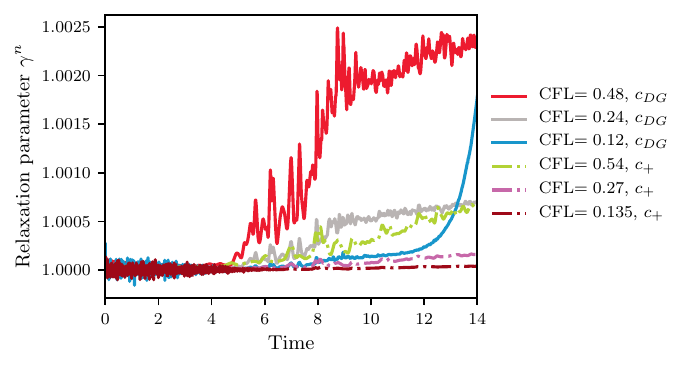}
    \caption{Evolution of the relaxation parameter in inviscid TGV test case when FD-NSFR is used.}
    \label{fig:TGV-relaxation-param}
\end{figure}

We expect the relaxation parameter to scale at $1+\mathcal{O}(\Delta t^{p-1})$. 
Once the flow loses its structured nature, it is visually apparent in Fig.~\ref{fig:TGV-relaxation-param} that the relaxation parameter follows a consistent trajectory across different refinements. Therefore, we evaluate convergence of the relaxation parameter at a specified time. Table \ref{tab:relaxation_convergence_TGV} demonstrates that convergence orders are close to $2$ at the point $t=7.0$ for both $c_{DG}$ and $c_+$. Convergence was close to $2$ for most times in the vicinity of $t=5$ to $t=8$, and deviations from $2$ are expected with the amount of noise seen in Fig.\ref{fig:TGV-relaxation-param}.

\begin{table}[h]
\centering
\caption{Convergence of the relaxation parameter towards $1$ at $=7.0$ for the TGV test case.}
\label{tab:relaxation_convergence_TGV}
\begin{tabular}{|p{0.75 in}|r|r|r|}
\hline
FR Scheme & CFL &     $\gamma(t=7.0)$ &  $\gamma(t=7.0)$ rate \\
\hline
        & 0.48 & 1.0011446 &         - \\
$c_{DG}$& 0.24 & 1.0002907 &   1.98 \\
        & 0.12 & 1.0000729 &   2.00 \\
        & 0.06 & 1.0000181 &   2.01 \\
\hhline{|=|=|=|=|}
        & 0.54   & 1.0002573 &         - \\
$c_{+}$ & 0.27   & 1.0000650 &   1.99 \\
        & 0.135  & 1.0000173 &   1.91 \\
        & 0.0675 & 1.0000036 &   2.28 \\
\hline
\end{tabular}
\end{table}

We can see a clear benefit of FD-NSFR over SD-NSFR for inviscid cases by comparing the two schemes in an entropy-conserving manner. In such a comparison, we decrease the time step size of the SD-NSFR calculation until numerical entropy is conserved on the order of $10^{-11}$. Table \ref{tab:n-tsteps-entropy-conserving} demonstrates that when we compare on an entropy-conserving basis, FD-NSFR requires far fewer time steps. 
Here, we define entropy conservation based on the change in $L^2$ numerical entropy, using the indicator developed in section~\ref{sec:RRK-root}.
Other authors have reported an increased cost of about $1.5\times$ when comparing semi-discrete entropy-stable DGSEM and fully-discrete entropy-stable \cite{rogowski2022performance}.
We see a roughly similar cost increase when comparing at the same time step size.
{\color{black}We argue that a comparison at the same large time step size does not account for the stability benefits offered by fully-discrete entropy-stable methods. Hence, we present Table~\ref{tab:n-tsteps-entropy-conserving} as a supplemental comparison of simulations having equivalent stability properties.}

\begin{table}[hbt!]
\caption{Number of time steps for $c_{+}$ and $c_{DG}$ in the inviscid TGV test case with FD-NSFR and SD-NSFR. For SD-NSFR cases, the time step has been chosen such that numerical entropy changes on the order of $10^{-11}$.}
\label{tab:n-tsteps-entropy-conserving}
\centering
\begin{tabular}{|c|c|c||c|c|}
\hline
Scheme & RK Type & FR Type & CFL & Number of time steps \\
\hline
FD-NSFR & SSPRK3 &$c_{+}$ & 0.54 &  1454 \\
FD-NSFR & SSPRK3 & $c_{DG}$ & 0.48 & 1634 \\
SD-NSFR & SSPRK3 & $c_{+}$ & 0.0675 & 11622 \\
SD-NSFR & SSPRK3 & $c_{DG}$ & 0.005 & 192043 \\
\hline
\end{tabular}
\end{table}

\subsubsection{Application of RRK to a semidiscretization without an entropy stability guarantee does not result in an entropy-stable scheme.}
{\color{black} Next, the use of RRK for a non-entropy conserving spatial discretization is investigated, revealing that this application does not result in a fully discrete entropy-stable scheme.}
Figure \ref{fig:TGV-strong} presents the inviscid TGV test case similar to the preceding results, but using a non-split strong ESFR scheme using $c_{DG}$ and the Roe flux \cite{roe1981approximate} with Harten's entropy fix~\cite{harten1983self}. We observe that, in the absence of nonlinear stability or other stabilization techniques, the non-split scheme is unable to reach the end time at $t=14$. There is a non-physical increase in entropy, growing until the calculation crashes. Adding RRK to the scheme does not prolong the simulation's end time. In this context, RRK only addresses entropy change due to the temporal discretization. While the relaxation parameter behaves similarly to the FD-NSFR case for most of the run time, it does nothing to address the entropy violation of the spatial discretization. 
It may be tempting to force entropy conservation by setting the entropy change estimate term in Eq. \ref{eq:rrk-root} to be zero {\color{black} (even if this is clearly non-zero for a non-entropy conserving scheme)} therefore ensuring $\eta -\eta_0 = 0$ such that there is no entropy generation over a time step from either the spatial or temporal discretizations. While such a modification does result in entropy conservation, Figure~\ref{fig:TGV-strong} demonstrates that the relaxation parameter quickly departs from $1$ and will decrease until time no longer advances.

\begin{figure}[h!]
    \centering    
    \includegraphics[height = 2.5 in]{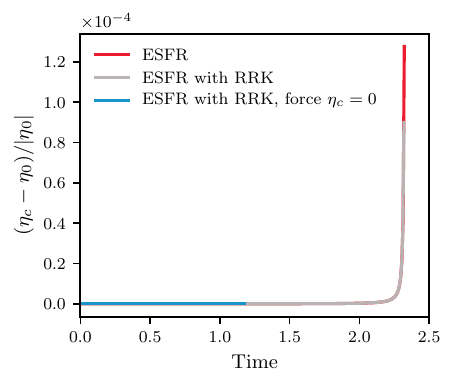}
    \includegraphics[height = 2.5 in]{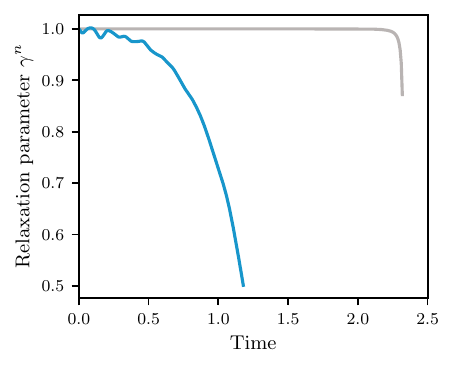}
    \caption{Entropy evolution (left) and relaxation parameter evolution (right) when applying RRK to ESFR. Entropy change is plotted relative to $t=0$.}
    \label{fig:TGV-strong}
\end{figure}

\subsection{Results using the inviscid Taylor-Green vortex test case} \label{sec:visc-tgv-results}
We will next move to the viscous TGV test case, using the Navier-Stokes equations which were defined in Section \ref{sec:NS-euler-governing-eqns}. We use the same TGV initial condition as in Eq.~(\ref{eq:tgv-initial-cond}), with the standard values $Re=1600$ and $Pr=0.71$ for air. The resulting flow begins as a laminar, structured flow, then as the vortices interact,  turbulent structures form and the flow eventually decays back to rest. We solve until $t_f=10$, which captures peak dissipation.


For the present results, we use $p=5$ and $16$ evenly-spaced elements per direction, corresponding to $96^3$ DOF per dimension. The discretization uses GLL quadrature points for the solution and GL quadrature points for the flux without overintegration.
The grid is straight-sided and is periodic in all coordinate directions.
We use Ranocha's modification to the Chandrashekar flux for the two-point convective flux~\cite{ranocha2021preventing, chandrashekar2013kinetic}, with the viscous and solution numerical fluxes defined using the symmetric interior penalty method~\cite{arnold1982interior} stated in Eq. (\ref{eq:SIP}).
We employ no turbulence model. 
Rather, we use the implicit LES (iLES) approach, where the truncation error associated with the grid is assumed to account for subgrid-scale dissipation~\cite{boris1992new,margolin2006modeling}. 
We use SSPRK3 for the RK method, and CFL-adaptation as described in the inviscid TGV section. Unless otherwise specified, we set the CFL numbers as a large, but stable, size for each FR correction parameter.

Figure \ref{fig:VTGV-SD-FD} demonstrates that FD-NSFR and SD-NSFR yield visually indistinguishable dissipation rates. We report kinetic energy-based dissipation, calculated using finite difference in time, 
\begin{equation}
    \epsilon_{K\hspace{-2pt}E} = -\frac{\text{d}\ K\hspace{-4pt}E}{\text{d}t},
\end{equation}
\noindent where kinetic energy per unit volume is integrated across the domain at $p+10$:
\begin{equation}
    K\hspace{-4pt}E = \frac{1}{\Omega}\int_\Omega \rho \frac{||\mathbf{v}||^2}{2} d \Omega.
\end{equation}
As the TGV test case is nearly incompressible at $M=0.1$, we can approximate the dissipation rate as \cite{wang2013high}
\begin{equation}
    \epsilon_{\omega} = \frac{2 \varepsilon}{Re},
\end{equation}
where enstrophy per unit volume $\varepsilon$ is found by integrating across the domain at $p+10$,
\begin{equation}
    \varepsilon = \frac{1}{\Omega}\int_\Omega  \rho \frac{||\bm{\omega}||^2}{2 } d \Omega,
\end{equation}
\noindent with vorticity calculated as $\bm{\omega}=\nabla \times \mathbf{v}$. We approximate the dissipation by the grid  as the difference between KE-based and vorticity-based dissipation, 
\begin{equation}
    \epsilon_{num} = \epsilon_{total}-\epsilon_{\omega}.
\end{equation}
Using the $96^3$-DOF viscous TGV test case with the $c_+$ correction parameter, we demonstrate that the SD-NSFR and FD-NSFR schemes both produce high-quality results when the time step is chosen to be relatively large. In fact, the two schemes yield dissipation evolutions which appear superimposed, which indicates that FD-NSFR does not impact the solution quality for this low-Mach turbulent flow.

\begin{figure}[h!]
    \centering
    \includegraphics[height = 2.5 in]{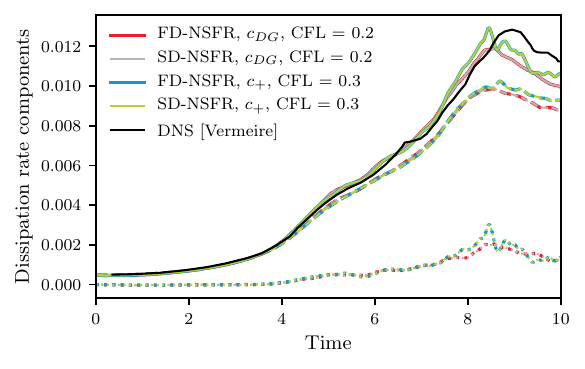}
    \caption{Dissipation rate components in the viscous TGV test case. The solid line is kinetic energy-based dissipation rate, $\epsilon_{K\hspace{-2pt}E}$, the dashed line is vorticity-based dissipation rate, $\epsilon_{\omega}$, and the dotted line is approximate numerical dissipation $\epsilon_{num}$. The DNS data are from Vermeire~\cite{vermeire2014adaptive}. The FD-NSFR and SD-NSFR lines nearly coincide.}
    \label{fig:VTGV-SD-FD}
\end{figure}

The NSFR entropy evolution $\eta_c$ is strictly decreasing for any $c$ value tested, demonstrated in Fig. \ref{fig:VTGV-num-entropy-Cparam} for $c_{DG}$, $c_{Hu}$ and $c_{+}$ using the FD-NSFR scheme. 
The $L^2$ numerical entropy evolution for each $c$ parameter is slightly different.
However, even for the extreme case of $c_+$, the NSFR numerical entropy change remains dissipative. 

The SD-NSFR and FD-NSFR schemes have very similar FR-corrected entropy change evolutions when comparing at the same $c$ value. 
{\color{black}Fig. \ref{fig:VTGV-num-entropy-Cparam} indicates that NSFR entropy change is on the order of $10^{-5}$ for any value of $c$.
The difference in NSFR entropy change between SD-NSFR and FD-NSFR is on the order of $10^{-11}$ using the same spatial scheme.
Thus, the dissipation of numerical entropy is dominated by physical viscosity and any change due to the time step size is comparatively small. }
From this perspective, for low-Mach, low-Reynolds turbulence simulations using a large time step size, entropy change from a standard RK method is unlikely to corrupt results unless the simulation is run for a very long time interval.

\begin{figure}
    \centering
    \includegraphics[height = 2.5 in]{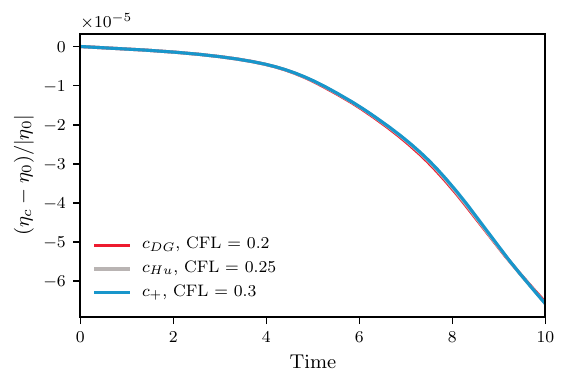}
    \caption{The evolution of NSFR numerical entropy in the viscous TGV test case for various FR correction parameters using the FD-NSFR scheme. Entropy change is plotted relative to the initial condition.}
    \label{fig:VTGV-num-entropy-Cparam}
\end{figure}

We next perform a time step refinement study to confirm convergence orders for the RRK method. 
Figure \ref{fig:VTGV-gamma-conv} shows the evolution of the relaxation parameter in the viscous TGV test case during a time refinement study. 
It is apparent that the relaxation parameter is quite noisy, especially at lower $t$. 
Especially during the laminar phase of the viscous TGV test case, there is only a small difference in the solution state between time steps, resulting in noise from subtractive cancellation errors.

\begin{figure}
    \centering
    \includegraphics[height = 2.5 in]{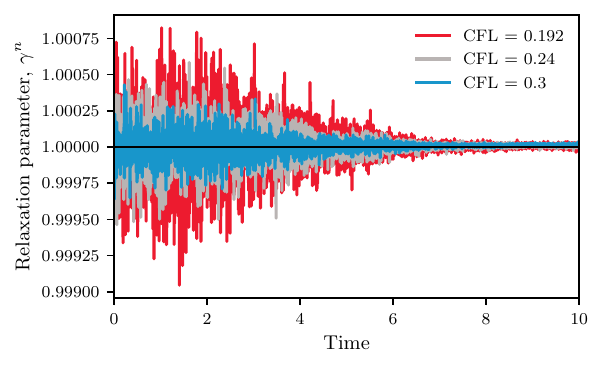}
    \caption{Evolution of the relaxation parameter during the viscous TGV test case.}
    \label{fig:VTGV-gamma-conv}
\end{figure}

We confirm that the relaxation parameter scales at $1+\mathcal{O}(\Delta t^{p-1})$ for the viscous TGV case by performing a temporal convergence study in Table \ref{tab:VTGV-gamma-conv} using a refinement ratio of $0.8$. The relaxation parameter is quite noisy even in the turbulent portion of the test case. Therefore, we report convergence of an average value of the relaxation parameter, after a steady-state value has been reached after $t=9$. the relaxation parameter is scaling as expected for both $c_{DG}$ and $c_+$, thus we are reassured that we maintain the desired temporal convergence order. 
{We note that the relaxation parameter reaches a steady-state value larger than $1.0$, indicating that the time step size is indeed being modified to preserve the expected entropy evolution. We conjecture that if the time step size is increased further -- for instance using implicit RK -- there will be a greater need for FD-NSFR.}
\begin{table}[h!]
    \centering
        \caption{Convergence of average $\gamma$ after $t=9$ using a refinement ratio of $0.8$. Convergence is evaluated toward $1$.}
        \begin{tabular}{|c|rr|}
        \hline
      CFL &     Average $\gamma$ &  Rate \\
      \hline
0.300 & 1.0000151 &           - \\
0.240 & 1.0000096 &   2.00 \\
0.192 & 1.0000062 &   1.98 \\
\hline
\end{tabular}
    \label{tab:VTGV-gamma-conv}
\end{table}

\section{Robustness of FD-NSFR} \label{sec:robustness}

In the preceding sections, we demonstrate that expected orders of convergence are maintained using FD-NSFR, and that the relaxation parameter does indeed modify the time step size to maintain the expected numerical entropy evolution. 
The benefit of FD-NSFR is clear in the inviscid TGV case, where numerical entropy dissipation is reduced from the order of $10^{-7}$ to machine precision. However, when moving to the viscous TGV case, viscous dissipation impacts the entropy evolution much more than temporal entropy change when using explicit RK. 
Therefore, we examine whether the cost of FD-NSFR is justified by the resulting improvement in solution quality for viscous cases.
We explore the failure modes of the SD-NSFR and FD-NSFR schemes using the TGV test case and the Kelvin-Helmholtz instability.
\subsection{Inviscid Taylor-Green vortex} \label{sec:inv-tgv}
We test whether RRK is capable of improving the robustness of FD-NSFR over SD-NSFR at a large time step size.
We find the largest stable time step for $c_{DG}$ to be governed by $CFL = 0.48$ for a simulation end time of $t_f=14$.
In this test, we extend the time step by gradually increasing the CFL number to $0.52$. While we expect that both schemes will fail at some point, Figure \ref{fig:TGV-unstable-CFL} demonstrates that the FD-NSFR scheme is able to extend the simulation end time.

When the CFL number is increased, the SD-NSFR scheme fails due to the formation of spurious oscillations within the domain. As shown in Figure~\ref{fig:TGV-unstable-failure-mode}, the relaxation parameter deviates from $1+\mathcal{O}(\Delta t^{p-1})$, a critical property essential for the proofs provided by Ketcheson and Ranocha \textit{et al.} ~\cite{ketcheson2019relaxation,ranocha2020relaxation}. This deviation can lead to unreliable temporal accuracy. Consequently, we conclude that the FD-NSFR method does not permit the use of larger time steps beyond the stability limit in the inviscid TGV test, though the extension of the end time suggests a robustness advantage.

\begin{figure}[h!]
    \centering
    \includegraphics[height = 2.5 in]{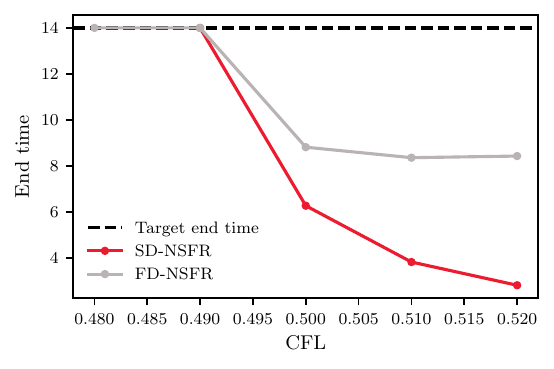}
    \caption{End time of inviscid TGV cases as the CFL number is increased. An the end time is less than $14$ indicates that the simulation crashed due to instability. Note when, $CFL=0.49$ both schemes had some instability forming near the end time, but neither crashed.}
    \label{fig:TGV-unstable-CFL}
\end{figure}

\begin{figure}[h!]
    \centering
    \includegraphics[height = 2.5 in]{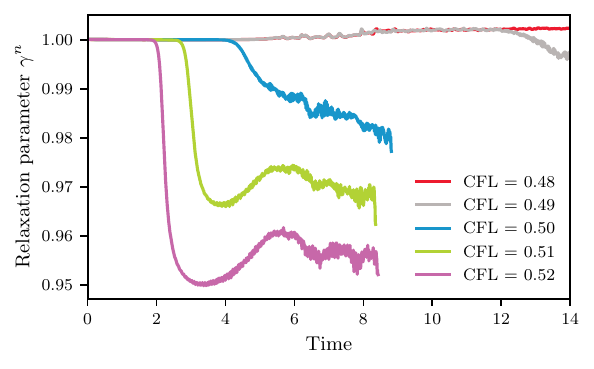}
    \caption{Relaxation parameter evolution for the inviscid TGV test case using FD-NSFR as the CFL number is increased past the stable size governed by CFL=$0.48$. }
    \label{fig:TGV-unstable-failure-mode}
\end{figure}

\subsection{Viscous Taylor-Green vortex}
We run a similar test using the viscous TGV case. In Section~\ref{sec:visc-tgv-results}, we use CFL$\ =0.3$ as a conservative estimate for a stable time step size. 
Here, we increase the CFL number from $0.3$ to $0.4$ by increments of $0.01$ for the $c_+$ correction parameter and from $0.2$ to $0.26$ for the $c_{DG}$ correction parameter.
The viscous dissipation term of the Navier-Stokes equations has a stabilizing effect on oscillations, so both FD-NSFR and SD-NSFR are able to reach the end time using any of the CFL numbers in the interval.
However, the SD-NSFR scheme fails to produce accurate results as the CFL number increases. Oscillations increase in magnitude in the right subplots of Fig. \ref{fig:VTGV-dissipation}.
On the other hand, FD-NSFR is able to produce nearly identical results for any CFL number tested. 
The same trends are reflected with $c_+$ in Fig. \ref{fig:VTGV-dissipation} (a) and $c_{DG}$ in Fig. \ref{fig:VTGV-dissipation} (b). 
Not only is FD-NSFR able to recover a high-quality dissipation rate, but Figure \ref{fig:VTGV-vel-contours} shows that FD-NSFR also reproduces a small-time-step solution better than SD-NSFR.
Applying a scheme with a fully-discrete entropy condition has prevented the solution from being corrupted by oscillations introduced by unphysical temporal entropy growth.

\begin{figure}[h!]
    \centering
    \begin{subfigure}[t]{\textwidth}
    \centering
    \includegraphics[height=2.5 in]{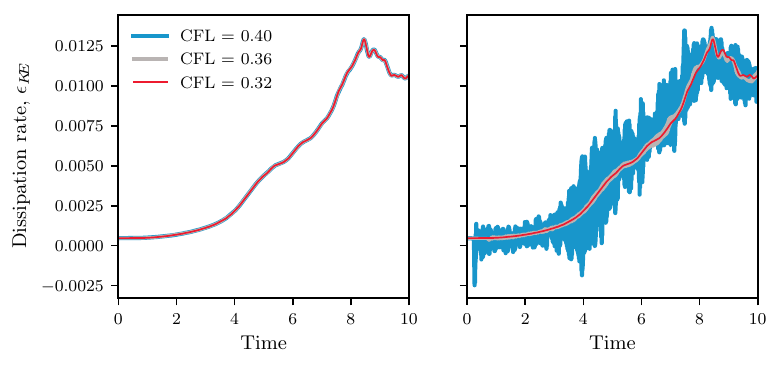}
    \caption{Dissipation rate using $c_+$.}
    \end{subfigure}
    \begin{subfigure}[t]{\textwidth}
    \centering
    \includegraphics[height=2.5 in]{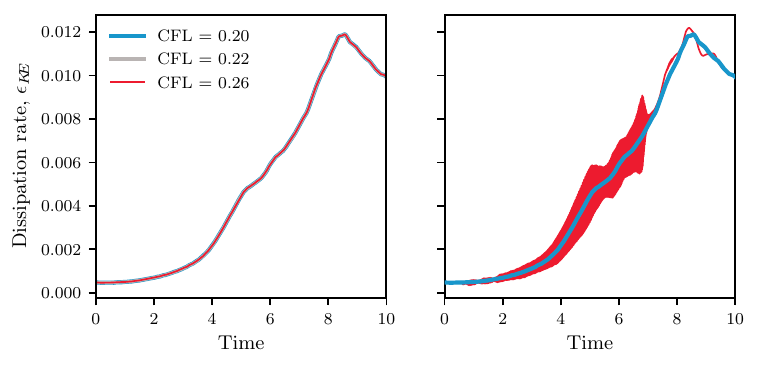}
    \caption{Dissipation rate using $c_{DG}$.}
    \end{subfigure}
    \caption{Kinetic energy-based dissipation rate as the CFL number is increased. The left figures use FD-NSFR and demonstrate that the same dissipation rate is recovered for any CFL number tested. In contrast, the right figures show that the SD-NSFR scheme yields an oscillatory solution for $CFL=0.36$ and $CFL=0.40$.}
    \label{fig:VTGV-dissipation}
\end{figure}

\begin{figure}
    \centering
        \begin{subfigure}[t]{0.48\textwidth}
    \includegraphics[width=\linewidth]{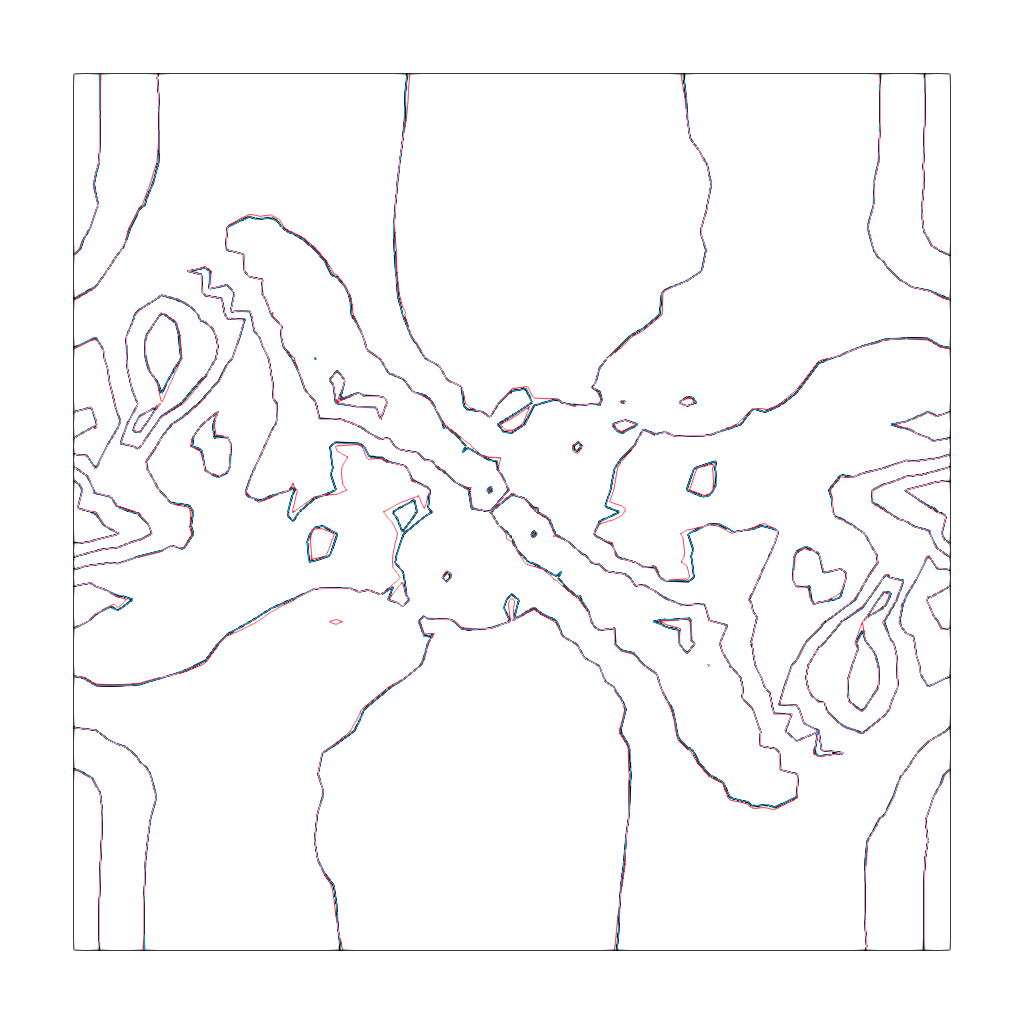}
    \caption{CFL=$0.35$}
    \end{subfigure}
    \begin{subfigure}[t]{0.48\textwidth}
        \includegraphics[width=\linewidth]{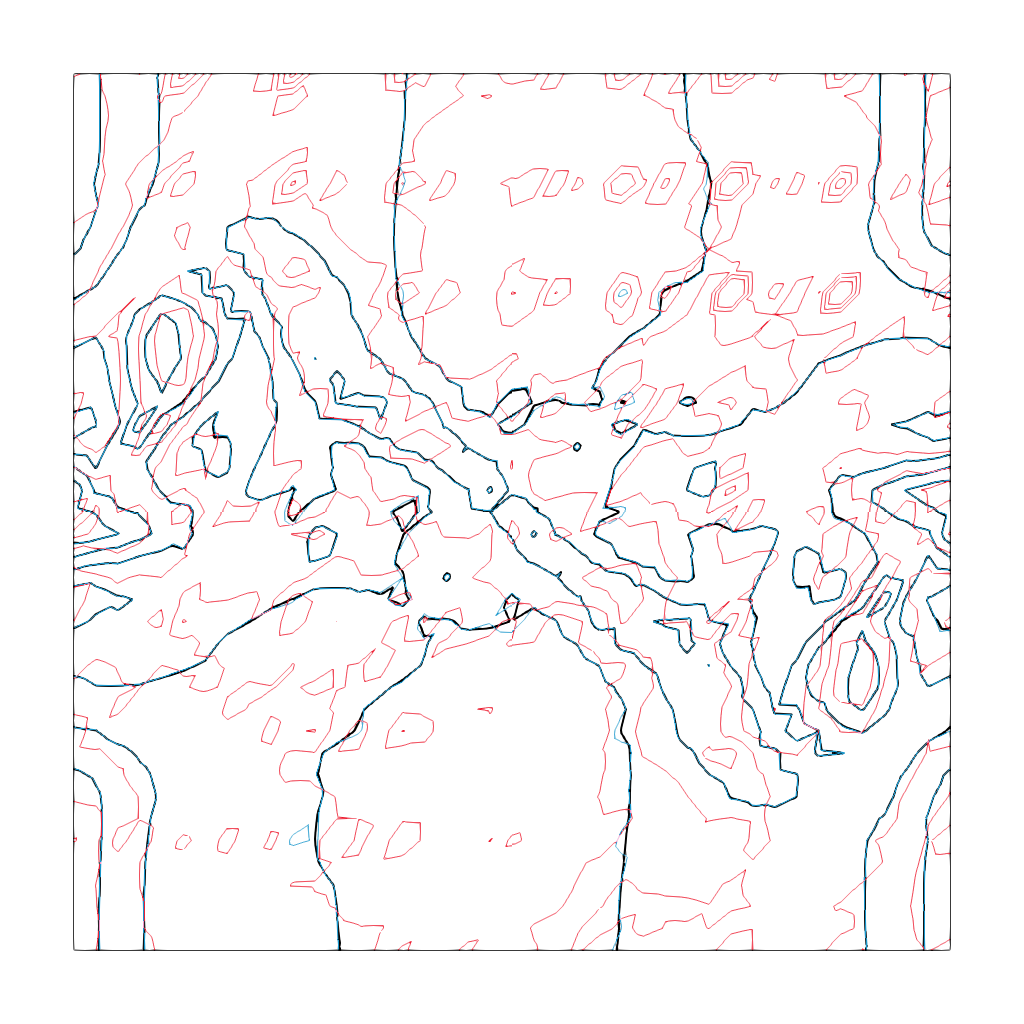}
        \caption{CFL=$0.40$}
    \end{subfigure}
    \caption{Contours of velocity magnitude for the viscous TGV test case using $c_+$. Red lines show SD-NSFR at the indicated CFL, blue lines show FD-NSFR at the indicated CFL, and black lines show SD-NSFR with a small CFL=$0.03$. At CFL$=0.35$, the blue and black lines are visually superimposed. Contours are at $\{0.2, 0.4, 0.6, 0.8\}$. We show a slice of the domain, $x \in [-\pi, 0], y = 0, z \in [-\pi, 0]$.}
    \label{fig:VTGV-vel-contours}
\end{figure}

We plot the relaxation parameter from the $c_+$ tests in Fig. \ref{fig:VTGV-gamma} to demonstrate that the relaxation parameter drops in order to reduce the time step size to a stable region. {When the time step is only slightly larger than the stable value, as seen in the $CFL=0.34$ and $CFL=0.35$ cases, the relaxation parameter successfully realigns with the trajectory of the smaller-CFL solutions, demonstrating that the solution is indeed recovered.}
The recovery mode observed in the viscous TGV test case was not seen in inviscid cases, so we hypothesize that viscosity is stabilizing the result sufficiently that the relaxation parameter is able to return to the vicinity of $1$.
We note that the temporal accuracy may degenerate due to the deviation from $1$; however, the similarity of the kinetic energy dissipation rate in Fig. \ref{fig:VTGV-dissipation} demonstrates that the overall solution accuracy is satisfactory.
The result is similar using the $c_+$ and $c_{DG}$ correction parameters, suggesting that correcting only the $L^2$ temporal entropy change is sufficient to achieve a good solution, despite the unaddressed temporal entropy change in the $\mathcal{K}$ norm.

\begin{figure}[h!]
    \centering
    \includegraphics[height = 2.5 in]{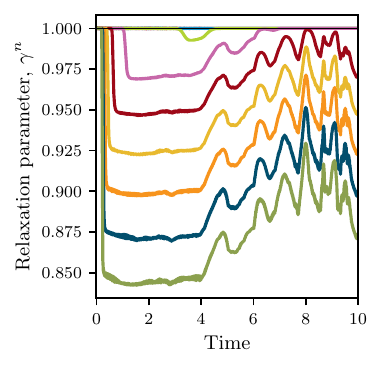}
    \includegraphics[height = 2.5 in]{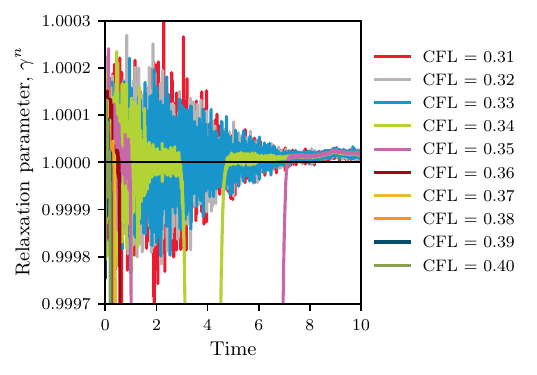}
    \caption{Relaxation parameter evolution using FD-NSFR as the CFL number is increased from $0.31$ to $0.40$ for the $c_+$ cases. The left figure zooms in to the vicinity of $1.0$ to demonstrate that the relaxation parameter is able to recover after deviating away from $1.0$.}
    \label{fig:VTGV-gamma}
\end{figure}

Figure \ref{fig:VTGV-time-step} demonstrates that the SD-NSFR and FD-NSFR cases result in a similar time step size magnitude as the CFL number increases.
 CFL adaptation is applied to both SD-NSFR and FD-NSFR. As presented in Algorithm \ref{alg:rrkForNavierStokes} for the FD-NSFR case, CFL adaptation adjusts the target time step size after applying the relaxation parameter $\gamma^n$.
In the case of SD-NSFR, time step size modifications as evident in Figure \ref{fig:VTGV-time-step}(right) resulted only from CFL adaptation. 
The precipitous drop observed between time 0.5 and 1 as the CFL is increased is due to the presence of high wavespeeds found within the domain, which are likely caused by spurious oscillations. The increasingly oscillatory solution is reflected in the contours of velocity magnitude, Fig.~\ref{fig:VTGV-vel-contours}.
On the other hand, the drop in time step size observed for FD-NSFR is only a result of RRK, where values of $\gamma^n < 1$ limit the temporal entropy change. Hence, even if the magnitude of the time step sizes is generally the same between the schemes, RRK has demonstrated the ability to yield a precise time step that ensures the solution tracks the expected trajectory as confirmed in Fig.~\ref{fig:VTGV-vel-contours}.

\begin{figure}[h!]
    \centering
    \includegraphics[height = 2.5 in]{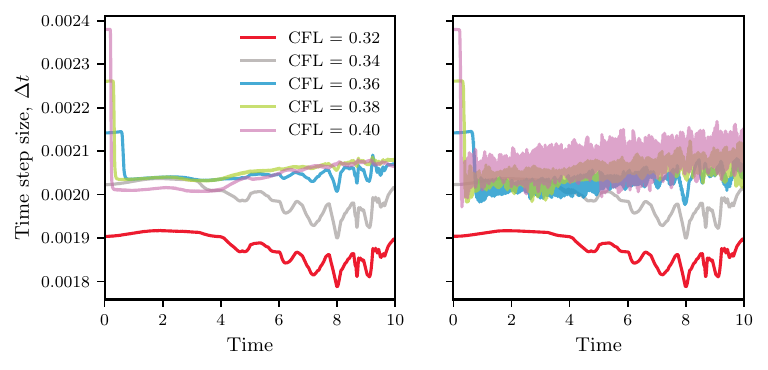}
    \caption{Time step size resulting from CFL adaptation and subsequent RRK modification in the FD-NSFR case (left) and from CFL adaptation in the SD-NSFR case (right) as the CFL number is changed for the $c_+$ case.}
    \label{fig:VTGV-time-step}
\end{figure}

\subsection{Kelvin-Helmholtz instability}
As a final test case, we assess the impact of the fully-discrete scheme on robustness in the presence of density gradients using the Kelvin-Helmholtz instability (KHI) test case as described by Chan \textit{et al.}~\cite{chan2022entropy}. The test case is defined for the Euler equations by the following initial condition on the periodic domain $\bm{x}\in[-1,1]$:

\begin{equation}
    \begin{aligned}
        \rho_1 &= 0.5 \\
        \rho_2 &= \rho_1 \frac{1+A}{1-A} \\
        B(x,y) &= \frac{1}{2} \left(  \tanh{(15y+7.5)} - \tanh{(15y-7.5)} \right) \\
        \rho(x,y) &= \rho_1 + B(x, y) * (\rho_2-\rho_1) \\
        v_1(x,y) &= B-\frac{1}{2} \\
        v_2(x,y) &= 0.1 \sin{(2\pi x)} \\
        p(x,y) &= 1.
    \end{aligned}
\end{equation}

As the KHI is a physically unstable case, it is expected to be challenging for numerical schemes. Furthermore, the presence of strong density gradients, which become stronger as vortical structures form, increases the risk of failure caused by negative density. To maintain consistency with~\cite{chan2022entropy}, we use Ranocha's modification to the Chandrashekar entropy-stable flux with Lax-Friedrichs dissipation to stabilize the results~\cite{ranocha2021preventing, chandrashekar2013kinetic}. We use RK4 with a very small time step size dictated by CFL$=0.01$. Reducing the time step size further does not change the results. {\color{black}We first provide a comparison of both a high- and low-$p$ solution at an equal number of degrees of freedom in Figure~\ref{fig:KHI-vis}, confirming that the high-$p$ version resolves finer flow structures. }

For a robustness test, we use a straight $16\times16$ grid with $p=7$, corresponding to~\cite[Figure 7(b)]{chan2022entropy}. We set a target end time $t=10$. If the simulation is able to reach the target end time, we deem it stable.
The schemes developed by Chan \textit{et al.}~\cite{chan2018discretely,chan2022entropy} and Cicchino \textit{et al.}~\cite{cicchino2022nonlinearly,cicchino2022provably,cicchino2023discretely} both use entropy-projected variables to evaluate fluxes, allowing uncollocated solution and flux nodes.
{\color{black}We provide a robustness comparison using two variations of node choices for the purpose of reproducing the findings of Chan \textit{et al.}~\cite{chan2022entropy} -- that is, selecting uncollocated solution and flux nodes prolong the simulation time.}
We present results for $c_{DG}$ {\color{black} on the left of Figure~\ref{fig:KHI-crash-times} and $c_{Hu}$ on the right of the same figure}. Figure~\ref{fig:KHI-crash-times} (left) nearly reproduces results in~\cite{chan2022entropy} when the same nodes are used. Discrepancies can be attributed to the use of a small time step rather than adaptive time-stepping and implementation differences. There is little difference between the two $c$ parameters, indicating that the choice of FR parameter does not have a large impact on the stability of this case. The correction parameter $c_+$ was omitted, as it is only listed to a maximum of $p=5$ by Castonguay~\cite{castonguay2012high}.

The results of Figure~\ref{fig:KHI-crash-times} show an imperceptible difference between FD-NSFR and SD-NSFR in the KHI case. We believe this is due to the very small time step employed in this case to prevent negative density. The robustness improvements seen in the TGV test cases were enabled by a large time size. 
{\color{black} When a large time step size was used with the KHI test case, FD-NSFR and SD-NSFR ended at similar times, indicating that the failure was not related to the temporal numerical entropy change.}
It is clear that fully-discrete entropy stability is not sufficient to mitigate negative density in the presence of strong density gradients.
    
\begin{figure}[h!]
    \centering    
    \begin{minipage}{0.5\textwidth}
        \centering
        \includegraphics[width=\textwidth]{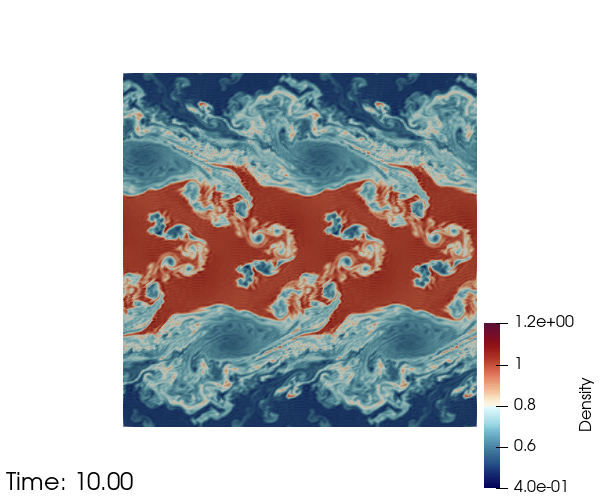} 
    \end{minipage}\hfill
        \begin{minipage}{0.5\textwidth}
        \centering
        \includegraphics[width=\textwidth]{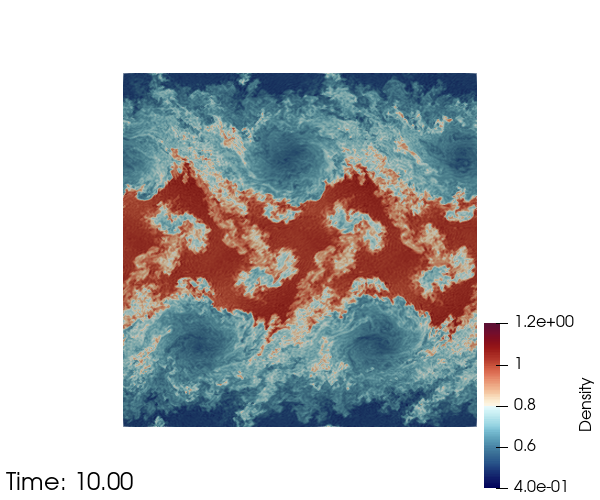} 
    \end{minipage}\hfill
    \caption{Two visualizations of the KHI test case at $A=1/3$ with an equal number of degrees of freedom. Left is $p=2, N = 128\times 128$, while the right is $p=11, N = 32\times32$. The higher-$p$ visualization shows more intricate flow structures.}
    \label{fig:KHI-vis}
\end{figure}

\begin{figure}[h!]
    \begin{minipage}{0.4\textwidth}
        \centering
        \includegraphics[height=2.5 in]{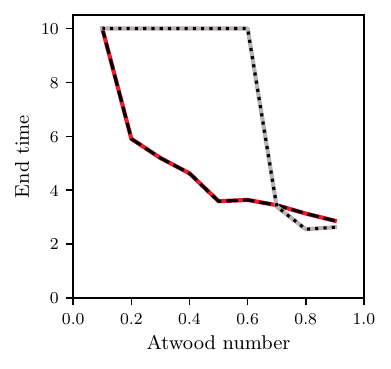} 
    \end{minipage}\hfill
        \begin{minipage}{0.6\textwidth}
        \centering
        \includegraphics[height=2.5 in]{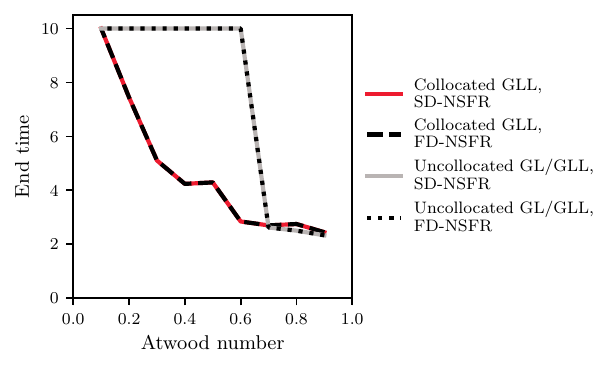} 
    \end{minipage}\hfill
    \caption{End times for the KHI simulation. An end time less than $10$ indicates that the simulation ended due to numerical instability. The left figure uses $c_{DG}$, and the right figure uses $c_{Hu}$.}
    \label{fig:KHI-crash-times}
\end{figure}

\section{Conclusions}
This work presents an FD-NSFR scheme and demonstrates its implementation using the inviscid and viscous Burgers' equations, the Euler equations, and the Navier-Stokes equations.
Crucially, we develop a method for implementing RRK in the broken Sobolev norm. 
For problems with inner-product numerical entropy, the method is fully-discretely entropy stable in the broken Sobolev $W^{k,2}_c$ norm. When numerical entropy is a general convex function, the temporal method prevents entropy change in the $L^2$ norm.
The FD-NSFR method enables us to benefit from the larger time step sizes allowed by FR schemes while still maintaining spatial and temporal entropy stability. 
We use a variety of inviscid and viscous test cases to show that numerical entropy evolves as expected while maintaining expected orders of convergence.
We demonstrate that FD-NSFR improves the robustness of low-Mach turbulence simulations at a large time step size. 
In iLES simulations of the viscous TGV test case, the FD-NSFR scheme enabled the solution to closely match a reference solution, even when using a higher CFL number than what would typically replicate the reference solution with the SD-NSFR scheme.
Our results suggest improved robustness when the exact stability limit is not known; however, this comes at the added expense of estimating the relaxation parameter. 
The FD-NSFR scheme does not prove beneficial in the KHI test case due to the small time step size needed to maintain positivity, motivating the addition of other robustness measures alongside entropy stability.
We found that FD-NSFR was the most advantageous when applied to problems using a very large time step size.
Applying RRK may be even more important to be used alongside implicit temporal integration to avoid temporal entropy accumulation which is on a comparable level to physical, viscous dissipation.

\section{Acknowledgements}
The authors are grateful to Alexander Cicchino for insights on the semidiscretization and to Julien Brillon for implementing the Taylor-Green vortex test case in \lstinline{PHiLiP}.
We acknowledge the support of the Natural Sciences and Engineering Research Council of Canada (NSERC) Discovery Grant Program [RGPIN-2019-04791] and Canadian Graduate Scholarships - Doctoral program [CGS-D-579552], and the support of McGill University. This research was enabled in part by support provided by Calcul Quebec and the Digital Research Alliance of Canada.

\bibliographystyle{elsarticle-num}
\bibliography{refs.bib}

\end{document}